\documentclass[11pt,a4paper]{article}
\usepackage{amssymb,amsmath,mathrsfs,enumerate,comment}
\usepackage{times,theorem,latexsym,color}
\allowdisplaybreaks[1]
\numberwithin{equation}{section}

\usepackage[colorlinks=true, pdfstartview=FitV, linkcolor=blue, citecolor=blue, urlcolor=blue,pagebackref=false]{hyperref}

\usepackage{tikz}
\usepackage{float}
\usepackage[font=footnotesize]{caption}

\newcommand{\BOX}{\ensuremath\Box}

\newtheorem{theorem}{Theorem }[section]

\newtheorem{corollary}[theorem]{Corollary}

{\theorembodyfont{\rmfamily}}
{\theorembodyfont{\rmfamily}}
{\theorembodyfont{\rmfamily}}
\newtheorem{lemma}[theorem]{Lemma}
\newtheorem{proposition}[theorem]{Proposition}
{\theorembodyfont{\rmfamily}\newtheorem{remark}[theorem]{Remark}}
{\theorembodyfont{\rmfamily}}

\newcommand{\Z}{\mathbb{Z}}
\newcommand{\R}{\mathbb{R}}
\newcommand{\C}{\mathbb{C}}
\newcommand{\T}{\mathbb{T}}
\newcommand{\dd}{{\rm d}}
\newcommand{\opdiv}{\operatorname{div}}

\newcommand{\oprot}{\operatorname{rot}}

\newcommand{\opM}{\operatorname{M}}
\newcommand{\opX}{\operatorname{X}}
\newcommand{\opFM}{\operatorname{FM}}
\newcommand{\opFX}{\operatorname{FX}}
\newcommand{\opBUC}{\operatorname{BUC}}

\newcommand{\iu}{\mathrm{i}\mkern1mu}

\DeclareMathOperator*{\esssup}{ess\,sup}

\def\XXint#1#2#3{{\setbox0=\hbox{$#1{#2#3}{\int}$}
		\vcenter{\hbox{$#2#3$}}\kern-.5\wd0}}

\newenvironment{proof}{{\vskip\baselineskip\noindent\textbf{Proof:}}}%
{\hspace*{.1pt}\hspace*{\fill}\BOX\vskip\baselineskip}

\newenvironment{proofx}[1]%
{\vskip\baselineskip\noindent\textbf{Proof of {#1}:}}%
{\hspace*{.1pt}\hspace*{\fill}\BOX\vskip\baselineskip}
{\vskip\baselineskip\noindent\textbf{Proof of Theorem \protect\ref{#1}:}}%
{\hspace*{.1pt}\hspace*{\fill}\BOX\vskip\baselineskip}
{\vskip\baselineskip\noindent\textbf{Proof of Theorems \protect\ref{#1} --
		\protect\ref{#2}:}}%
{\hspace*{.1pt}\hspace*{\fill}\BOX\vskip\baselineskip}

\begin{document}

\title{
Axisymmetric steady Navier-Stokes flows under suction
}

\author{Mitsuo Higaki \\
Department of Mathematics, Graduate School of Science, Kobe University, \\
1-1 Rokkodai, Nada-ku, Kobe 657-8501, Japan \\
E-mail: higaki@math.kobe-u.ac.jp}
\date{}

\maketitle

\noindent {\bf Abstract.}\ 
We prove the existence of solutions for the axisymmetric steady Navier-Stokes system around an infinite cylinder under external forces. The solutions are constructed to be decaying at the horizontal infinity, despite an analogue of the Stokes paradox for the linearized system, and having neither periodicity nor decay in the vertical direction. The proof is based on perturbation of the nonlinear system around a suction flow. The class of functions in this paper, which is a subspace of the space of Fourier transformed vector finite Radon measures, is inspired by Giga-Saal (2013) treating rotating boundary layers.

\tableofcontents

%%%%%%%%%%%%%%%%%%%%
%%%%%%%%%%%%%%%%%%%%
\section{Introduction}
\label{sec.intro}
%%%%%%%%%%%%%%%%%%%%
%%%%%%%%%%%%%%%%%%%%

We consider the steady Navier-Stokes system in $\Omega$ 
\begin{equation}\tag{NS}\label{intro.eq.NS}
\left\{
\begin{array}{ll}
-\Delta u + \nabla p 
= -u\cdot\nabla u + f&\mbox{in}\ \Omega \\
\opdiv u = 0&\mbox{in}\ \Omega \\
u= b &\mbox{on}\ \partial\Omega \\
u(x,z)\to0&\mbox{as}\ |x|\to\infty. 
\end{array}\right.
\end{equation}
Here $\Omega=D\times\R$ and $D=\{x=(x^1,x^2)\in\R^2~|~|x|>1\}$ denotes a two-dimensional exterior disk. Thus $\Omega$ represents the domain exterior to an infinite straight cylinder. The unknowns are the three-dimensional velocity field $u=(u^1,u^2,u^z)$ and the pressure field $p$. The given data are the external force $f$ and the boundary condition $b$ on $\partial\Omega$. We have used standard notation for derivatives: 
$\partial_j = \partial/\partial x^j$, 
$\partial_z = \partial/\partial z$, 
$\Delta = \sum_{j=1}^2 \partial^2_j + \partial_z^2$, 
$\nabla = (\partial_1, \partial_2, \partial_z)$, 
$u\cdot \nabla u = (\sum_{j=1}^2 u^j \partial_j + u^z\partial_z) u$ 
and 
$\opdiv u = \sum_{j=1}^2 \partial_j u^j + \partial_z u^z$.

In this paper, we discuss the existence of solutions for the system \eqref{intro.eq.NS}. A notable property of \eqref{intro.eq.NS} is that, for given data of a specific form, it is reduced to the Navier-Stokes system in the two-dimensional exterior domain $D$. Indeed, if the data $f,b$ are independent of the variable $z$ and $f^z=b^z=0$, the system \eqref{intro.eq.NS} is equivalent with 
\begin{equation}\tag{NS$_{\rm 2d}$}\label{intro.eq.NS.2d}
\left\{
\begin{array}{ll}
-\Delta_{\rm 2d} w + \nabla_{\rm 2d} s 
= -w\cdot\nabla_{\rm 2d} w + (f_1,f_2)&\mbox{in}\ D \\
\opdiv_{\rm 2d} w = 0&\mbox{in}\ D \\
w = (b_1,b_2) &\mbox{on}\ \partial D \\
w(x)\to0&\mbox{as}\ |x|\to\infty, 
\end{array}\right.
\end{equation}
of which the unknowns are $w=(w_1,w_2)$ and $s$. The symbols with the subscript ``{\rm 2d}" refer to the two-dimensional operators on $D$ corresponding to the three-dimensional ones on $\Omega$. It is well known that, for general data $f,b,D$, the existence and uniqueness of solutions for \eqref{intro.eq.NS.2d} is widely open. There are difficulties due to lack of embedding in two-dimensional unbounded domains and to the famous Stokes paradox \cite{ChaFin1961,KozSoh1992,Gal04,Rus2010,Galbook2011} caused by logarithmic spatial growth of the fundamental solution of the linearized system. We refer to a recent survey \cite{KorRen2023} for an overview of the literature and challenging questions left open.

These aspects are deeply related to controlling the behavior of weak solutions far away from the boundaries, and hence, the zero condition at spatial infinity in \eqref{intro.eq.NS.2d} requires careful handling. There are two methods that guarantee the existence of solutions. The one is to impose symmetry on given data; see \cite{Gal04,Rus09,Yam11,Galbook2011,PilRus12,Yam16,Yam18}
for the existence of solutions and \cite{Nak15,Yam18} for the uniqueness. The other is to perturb the system. Taking a $(-1)$–homogeneous smooth solution of the Navier-Stokes system in $\R^2\setminus\{0\}$, we perturb \eqref{intro.eq.NS.2d} around it for a suitable boundary condition $(b_1,b_2)$ to obtain a new nonlinear problem. As a prototypical example, denoting by $(r,\theta)$ the polar coordinates on $D$
\begin{align}\label{intro.plr.coord.}
\begin{split}
&x_1 = r\cos \theta, 
\quad
x_2 = r\sin \theta, 
\quad 
r= |x| \ge 1, \quad \theta\in [0,2\pi), 
\end{split}
\end{align}
we may take the following exact solution, which is also one of the Hamel solutions \cite{Ham1917}: 
$$
{\mathcal V}_{\alpha,\gamma}(r,\theta) 
= 
\frac{\alpha}{r} (-\sin\theta, \cos\theta)
- \frac{\gamma}{r} (\cos\theta, \sin\theta), 
\quad 
\gamma,\alpha\in\R.
$$
This is a linear combination of a rotating flow and a flow carrying flux $-2\pi\gamma$; see \cite{Sve2011,GuiWit2015} for the characterization of the Hamel solutions as solutions to the Navier-Stokes system in $\R^2\setminus\{0\}$. The idea of the perturbation method is to use scalar transport by ${\mathcal V}_{\alpha,\gamma}$, whose effect is expected to localize fluid flow and to improve spatial decay of vorticity. Once such an improved decay is obtained, we associate with it the velocity field by using the streamfunction. The known results along this approach are summarized as follows.
\begin{itemize}
\item
The case with $\gamma=0$ is considered in Hillairet-Wittwer \cite{HilWit2013}. The transport by
\begin{align}\label{intro.trans.rot}
\begin{split}
{\mathcal V}_{\alpha,0} \cdot \nabla_{\rm 2d} \omega 
= \frac{\alpha}{r^2} \partial_\theta \omega
\end{split}
\end{align}
improves spatial decay of solutions to the linearized problem. Moreover, when $|\alpha|$ is sufficiently large, more precisely, when $|\alpha|>\sqrt{48}$, the existence of solutions is proved in the form of $w = {\mathcal V}_{\alpha,0} + o(r^{-1})$ as $r\to\infty$. These solutions are driven by Dirichlet boundary data on which no symmetry is imposed. This highlights the difference between \cite{HilWit2013} and the works above under symmetry \cite{Gal04,Rus09,Yam11,Galbook2011,PilRus12,Yam16,Yam18}. We refer to \cite{MaeTsu2023} for a related result in the two-dimensional whole space $\R^2$.

\item
The case with $\gamma\neq0$ is considered in \cite{Hig2023}. The transport by
\begin{align}\label{intro.trans.suc}
\begin{split}
{\mathcal V}_{\alpha,\gamma}\cdot \nabla_{\rm 2d} \omega 
= \frac{\alpha}{r^2} \partial_\theta \omega 
- \frac{\gamma}{r} \partial_r \omega 
\end{split}
\end{align}
improves the spatial decay as in \cite{HilWit2013}. The existence of solutions is also obtained if $\gamma$ is sufficiently large, more precisely, if $\gamma>2$ in the class $w = {\mathcal V}_{\alpha,\gamma} + o(r^{-1})$ as $r\to\infty$. Unlike \cite{HilWit2013}, however, the existence is verified for all $\alpha\in\R$. Therefore, the existence essentially relies on the presence of a pure suction flow ${\mathcal V}_{0,\gamma}$. This result is extended in \cite{HigHor2023} to the system \eqref{intro.eq.NS} under the assumption that all data are independent of the variable $z$, which makes the situation similar to that in \cite{Hig2023}.
\end{itemize}

The aim of this paper is to provide an existence theorem of axisymmetric solutions for \eqref{intro.eq.NS} by developing the above perturbation method. Recall that a vector field $u$ on $\Omega$
$$
u
= u^r(r,\theta,z) {\bf e}_r 
+ u^\theta(r,\theta,z) {\bf e}_\theta 
+ u^z(r,\theta,z) {\bf e}_z,
$$
where $(r,\theta,z)$ is the cylindrical coordinates on $\Omega=D\times\R$ induced by \eqref{intro.plr.coord.} and ${\bf e}_r, {\bf e}_\theta, {\bf e}_z$ are respectively the unit vectors in the radial, toroidal, and vertical directions:
$$
{\bf e}_r = (\cos\theta, \sin\theta, 0), 
\quad 
{\bf e}_\theta = (-\sin\theta, \cos\theta, 0), 
\quad 
{\bf e}_z = (0, 0, 1), 
$$
is said to be axisymmetric if $u^r, u^\theta, u^z$ are independent of $\theta$. The corresponding unsteady system of \eqref{intro.eq.NS} has been widely studied in the whole space $\R^3$. In the absence of swirl, namely when $u^\theta=0$, the global well-posedness is established by Lady\v{z}enskaja \cite{Lad1968} and Ukhovksii-Yudovich \cite{UkhYud1968}. The asymptotic behavior of solutions in large time is obtained by Gallay-\v{S}ver\'{a}k \cite{GalSve2015}. The global well-posedness in the presence of swirl is an open question and there are extensive works on the regularity of solutions; see for instance \cite{CSYT2008,CSTY2009,KNSS2009,LeiZha2011} and \cite{LiuZha2018,Liu2023} for recent results on the global solutions. The steady system \eqref{intro.eq.NS} has been studied in various domains in connection to Liouville-type theorems; see for instance \cite{CSYT2008,KNSS2009,LeiZha2011,KPR2015,Ser2016,CPZZ2020,Tsa2021,KTW2022,KTW2023}. Among others, Kozono-Terasawa-Wakasugi \cite{KTW2023} studies Liouville-type theorems in $\Omega=D\times\R$ assuming periodicity in the vertical direction.

In contrast, the existence of solutions for \eqref{intro.eq.NS} is less investigated; see \cite{Wan1974,Oka1995} and the references therein for the systems in which the zero condition at horizontal infinity is replaced by the boundary conditions with growth in $(r,z)$ at spatial infinity. There is an inherent difficulty analogous to the Stokes paradox \cite{ChaFin1961,KozSoh1992,Gal04,Rus2010,Galbook2011}. Indeed, for given $f,b$ independent of $z$ and a solution $u$ of \eqref{intro.eq.NS}, the component $u^z$ must be subject to 
$$
-\Delta_{{\rm 2d}} u^z
= -\Delta u^z
= -(u\cdot\nabla u)^z + f^z.
$$
Thus the two-dimensional Laplacian $\Delta_{{\rm 2d}}$ appears, whose fundamental solution has logarithmic growth in space. This cannot be avoided also in an axisymmetric setting. Therefore, even if the given data $f,b$ have specific forms, the solvability of \eqref{intro.eq.NS} is a subtle issue.

To overcome this difficulty, we adapt the method in \cite{HilWit2013,Hig2023}. However, since the interaction in fluid is more complicated in three-dimensional cases, it is completely nontrivial whether the spatial decay of solutions can be improved by the perturbation. Also, in axisymmetric settings, such improvement cannot be expected from rotating flows around the $z$-axis, as suggested by \eqref{intro.trans.rot}. We thus take a pure suction flow towards the wall $\partial\Omega$
\begin{align}\label{def.V}
\begin{split}
V(r,\theta,z)
=V(r,\theta)
=-\frac{\gamma}{r} {\bf e}_r, 
\quad 
\gamma>0
\end{split}
\end{align}
and perturb around $V$ the Navier-Stokes system \eqref{intro.eq.NS} with the boundary condition 
$$
b = (-\gamma x,0) = -\gamma {\bf e}_r. 
$$
Assuming the form of solutions $u=v+V$ and using the formula
$$
v\cdot\nabla V + V\cdot\nabla v 
= -V\times\oprot v
+ \nabla\Big(
\frac{|v+V|^2-|v|^2-|V|^2}2
\Big), 
$$
where $\times$ denotes the cross product in $\R^3$, we derive the nonlinear problem to be solved 
\begin{equation}\label{intro.eq.PNP}
\left\{
\begin{array}{ll}
-\Delta v - V\times\oprot v + \nabla q 
= -v\cdot\nabla v + f
&\mbox{in}\ \Omega \\
\opdiv v = 0&\mbox{in}\ \Omega \\
v= 0 &\mbox{on}\ \partial\Omega \\
v(x,z)\to0&\mbox{as}\ |x|\to\infty. 
\end{array}\right.
\end{equation}
By the formulas in Subsection \ref{subsec.op}, we write \eqref{intro.eq.PNP} in the cylindrical coordinates as
\begin{equation}\tag{NP}\label{intro.eq.NP}
\left\{
\begin{array}{ll}
-\Big(
\Delta - \dfrac{1}{r^2} 
\Big) v^r 
+ \partial_r q 
= 
-v\cdot\nabla v^r
+ \dfrac{(v^\theta)^2}{r}
+ f^r
&\mbox{in}\ (1,\infty)\times\R \\ [10pt]
-\Big(\Delta
+ \dfrac{\gamma}{r} \partial_r
- \dfrac{1-\gamma}{r^2}
\Big) v^\theta 
=
-v\cdot\nabla v^\theta
- \dfrac{v^r v^\theta}{r}
+ f^\theta
&\mbox{in}\ (1,\infty)\times\R \\ [10pt]
-\Delta v^z 
+ \dfrac{\gamma}{r} (\partial_z v^r - \partial_r v^z)
+ \partial_z q
=
-v\cdot\nabla v^z 
+ f^z
&\mbox{in}\ (1,\infty)\times\R \\ [10pt]
\dfrac{1}{r} \partial_r (rv^r) 
+ \partial_z v^z 
= 0
&\mbox{in}\ (1,\infty)\times\R \\ [10pt]
(v^r,v^\theta,v^z) = (0,0,0)
&\mbox{on}\ \{1\}\times\R \\ [5pt]
(v^r,v^\theta,v^z) \to (0,0,0)
&\mbox{as}\ r\to\infty,
\end{array}\right.
\end{equation}
where $\Delta=\partial_r^2 + \frac{1}{r} \partial_r + \partial_z^2$ and $v\cdot\nabla=v^r\partial_r+v^z\partial_z$.

Notice that the problem \eqref{intro.eq.NP} is imposed in the vertically unbounded domain $\Omega\simeq(1,\infty)\times\R$, whereas the two-dimensional one \eqref{intro.eq.NS.2d} is in $D\simeq(1,\infty)\times\T$. Given the development of the results \cite{HilWit2013,Hig2023}, it is meaningful to study \eqref{intro.eq.NP} for external forces $f$ that are neither periodic nor decaying in $\R$. The function space to which $f$ belongs is desirable to include periodic functions in the vertical direction. Thus, with reference to Giga-Saal \cite{GigSaa2013} treating rotating boundary layers, we work in a space of Fourier transformed finite vector Radon measures $\opFM$. In such a space, a function $f(r,z)$ on $(1,\infty)\times\R$ periodic in $z$ can be viewed as the Fourier transformation of a sum of Dirac measures, with $r$-variable coefficients, on $\R$ supported at integer points. However, the proofs require some technical advancements, since the results \cite{HilWit2013,Hig2023} make use of mode-by-mode computation by the Fourier series expansion on $\T$. It is emphasized that lack of spatial decay along boundaries in fluid flow is a natural concept in the analysis of boundary layers; see \cite{APV1998,JagMik2001,BasGer2008,GerMas2010,HigZhu2023} for fluids over rough boundaries and \cite{Ger2003,GIMMS2007,GigSaa2013,FisSaa2013,DalPra2014,GigSaa2015,DalGer2017} for rotating fluids.

To state the main result, we introduce function spaces. More details are summarized in Subsection \ref{subsec.func.sp}. For $\rho\ge0$, we define $L^\infty$ spaces on $(1,\infty)$ with weight
\begin{align}\label{intro.def.L}
\begin{split}
L^\infty_\rho = \{ f\in L^\infty(1,\infty)~|~ 
\|f\|_{L^\infty_\rho} 
<\infty \}, \qquad
\|f\|_{L^\infty_\rho} := \esssup_{r\in(1,\infty)}\, r^\rho |f(r)|. 
\end{split}
\end{align}
Let $L^1(\R,L^\infty_\rho)$ denote the $L^\infty_\rho$-valued Lebesgue space with respect to the Lebesgue measure $\Lambda$, and $l^1(\Z,L^\infty_{\rho})$ the space of $L^\infty_\rho$-valued absolutely summable sequences. Denoting by $\delta_p$ the Dirac measure on $\R$ supported at the point $p$, we define the space
\begin{align}\label{intro.eq.X}
\begin{split}
\opX(\R,L^\infty_{\rho})
= \bigg\{
f=g \dd \Lambda
+ \sum_{m\in\Z} a_m \delta_m
~\bigg|~
g\in L^1(\R,L^\infty_{\rho}), 
\mkern9mu 
\{a_m\}\in l^1(\Z,L^\infty_{\rho})
\bigg\}
\end{split}
\end{align}
consisting of $L^\infty_\rho$-valued finite Borel measures on $\R$, equipped with the norm 
\begin{align}\label{intro.eq.X.norm}
\begin{split}
\|f\|_{\opX(\R,L^\infty_{\rho})}
:=
\|g\|_{L^1(\R,L^\infty_{\rho})} 
+ \sum_{m\in\Z} \|a_m\|_{L^\infty_{\rho}}. 
\end{split}
\end{align}
It is a closed subspace of the space of $L^\infty_\rho$-valued finite Radon measures $\opM(\R,L^\infty_{\rho})$ on $\R$. Denoting by $\hat{f}$ the Fourier transform of $f$, we define the space
\begin{align}\label{intro.eq.FX}
\begin{split}
\opFX(\R,L^\infty_{\rho})
= \{\hat{f}~|~f\in \opX(\R,L^\infty_{\rho})\}
\end{split}
\end{align}
equipped with the norm $\|\hat{f}\|_{\opFX(\R,L^\infty_{\rho})}=\|f\|_{\opX(\R,L^\infty_{\rho})}$. 
It is a closed subspace of the space of $L^\infty_\rho$-valued bounded continuous functions $\opBUC(\R,L^\infty_{\rho})$. Finally, we define
\begin{align*}
\begin{split}
\opFX(\R,L^\infty_{\rho})^3
= \{
f = f^r {\bf e}_r + f^\theta {\bf e}_\theta + f^z {\bf e}_z
~|~
f^r, f^\theta, f^z\in \opFX(\R,L^\infty_{\rho})
\}. 
\end{split}
\end{align*}
Notice that any element of $\opFX(\R,L^\infty_{\rho})^3$ is an axisymmetric vector field on $\Omega$.

The main result in this paper is the existence of solutions of \eqref{intro.eq.NS}. 
%
%%%%%%%%%%
\begin{theorem}\label{thm.main}
Let $\gamma>2$ and $2<\rho<3$ satisfy $\rho\le\gamma$. If $f\in \opFX(\R,L^\infty_{\rho+1})^3$ 
and $\|f\|_{\opFX(L^\infty_{\rho+1})}$ is sufficiently small depending on $\gamma,\rho$, then there is a unique weak solution $v\in \opFX(\R,L^\infty_{\rho-1})^3$ of \eqref{intro.eq.NP}. Consequently, there is an axisymmetric weak solution $u$ of \eqref{intro.eq.NS} with $b=-\gamma {\bf e}_r$ unique in a suitable class satisfying
\begin{align*}
u(r,\theta,z) 
= 
- \frac{\gamma}{r} {\bf e}_r
+ O(r^{-\rho+1})
\quad 
\mbox{as}\ \, r\to\infty.
\end{align*}
\end{theorem}
%%%%%%%%%%
%
%
%%%%%%%%%%
\begin{remark}\label{rem.thm.main}
\begin{enumerate}[(i)]
\item\label{item1.rem.thm.main}
To the best of our knowledge, the existence theorem in Theorem \ref{thm.main} is new even for vertically periodic data $f$. If such periodicity is assumed, the proofs in this paper are simplified considerably, since the Fourier series expansion is applicable.

\item\label{item2.rem.thm.main}
It is unclear whether the class of external forces $f$ in Theorem \ref{thm.main} can be extended to the entire $\opFM(\R,L^\infty_{\rho+1})^3$. This is related to the Radon-Nikod\'{y}m property \cite[Chapter $\mathrm{I}\hspace{-1.2pt}\mathrm{I}\hspace{-1.2pt}\mathrm{I}$, Definition 3]{DieUhlbook1977} of the space $L^\infty_{\rho+1}$. This point will be discussed in Remark \ref{rem.lem.X.FX}.

\item\label{item3.rem.thm.main}
In the presence of a rotating flow, namely, changing $V$ in \eqref{def.V} into 
\begin{align}\label{def.V.rot}
\begin{split}
V(r,\theta,z)
=V(r,\theta)
=\frac{\alpha}{r} {\bf e}_\theta
- \frac{\gamma}{r} {\bf e}_r, 
\quad 
\alpha\in\R, 
\mkern9mu
\gamma>0, 
\end{split}
\end{align}
one can ask whether an analogy of Theorem \ref{thm.main} holds. Nevertheless, we need smallness on $|\alpha|$ in this case to prove the existence of solutions by the method in this paper. This restriction is due to the interaction between the swirl $v^\theta$ and the rotating flow. Hence there is a clear difference between the existence theorem for \eqref{intro.eq.NS} and that for the two-dimensional system in \cite{Hig2023} in which no smallness on $|\alpha|$ is needed. The proof of the existence of solutions when $\alpha\neq0$ will be outlined in Appendix \ref{app.case.rot}.
\end{enumerate}
\end{remark}
%%%%%%%%%%
%

The proof of Theorem \ref{thm.main} is based on a representation formula for the linearized problem of \eqref{intro.eq.NP}; see \eqref{eq.LP} in Section \ref{sec.LP}. The formula is derived through the axisymmetric Biot-Savart law in $\Omega$, which seems to be new in the literature. If appropriate estimates of \eqref{eq.LP} are obtained, it is routine to prove the existence of solutions for \eqref{intro.eq.NP} by the fixed point argument under smallness on data. However, the proof of the linear estimates performed in Section \ref{sec.LP} requires tedious computation and is the most demanding part of this paper.

It is worth noting that the difficulties in this part have similarity to those in estimating at large time the semigroup for the Stokes system in two-dimensional exterior domains \cite{KozOga1993,BorVar1993,MarSol1997,DanShi1999a,DanShi1999b} or for the perturbed system around Hamel solutions \cite{Mae2017a,Mae2017b,Hig2019,Hig2024}. Such a large-time estimate is deduced from the corresponding resolvent problem when the resolvent parameter $\lambda$ is close to the origin, which requires careful analysis due to the Stokes paradox. In our case, the counterpart of $\lambda$ is $|\zeta|^2$ where $\zeta$ is the Fourier variable of $z$ in \eqref{eq.LP}, and the difficulties lie in estimating \eqref{eq.LP} when $|\zeta|$ is small. To address this, we utilize {\it monotonicity} in the representation formula, which is not available for the resolvent problems. In particular, we essentially use the positivity of the modified Bessel function $K_\nu(|\zeta|r)$ with positive order $\nu$; see the proofs of Lemma \ref{lem.vel.L1} and Lemma \ref{lem.est.F} for example.

This paper is organized as follows. Section \ref{sec.prelim} summarizes the differential operators in the cylindrical coordinates and the function spaces used in this paper. In Sections \ref{sec.ABSL}, we introduce the axisymmetric Biot-Savart law and its applications. Section \ref{sec.LP} is devoted to the study of the linearized problem of \eqref{intro.eq.NS}, which is the core of this paper, and Section \ref{sec.pr.thm} to the proof of Theorem \ref{thm.main}. Appendix \ref{app.lem} summarizes the auxiliary lemmas needed in Sections \ref{sec.ABSL} and \ref{sec.LP}. In Appendix \ref{app.case.rot}, we discuss the case where $V$ is not taken as \eqref{def.V} but as \eqref{def.V.rot}.

\smallskip

\noindent
{\bf Notation.}\
In the following, we write $A\lesssim B$ if there is some constant $C$, called the implicit constant, such that $A\le CB$, and moreover, analogously $A\gtrsim B$ if $B\lesssim A$ and $A\approx B$ if $A\lesssim B\lesssim A$. The dependence of an implicit constant on other parameters will be indicated contextually. We let $\langle \cdot, \cdot\rangle$ denote the usual pairing for mappings and $C^\infty_{0,\sigma}(\Omega)$ the space of smooth divergence-free vector fields on $\Omega$. If there is no confusion, we simplify the notation of norms: $\|\,\cdot\,\|_{\opX(\R,L^\infty_{\rho})}=\|\,\cdot\,\|_{\opX(L^\infty_{\rho})}$ and $\|\,\cdot\,\|_{\opFX(\R,L^\infty_{\rho})}=\|\,\cdot\,\|_{\opFX(L^\infty_{\rho})}$ for example.

%%%%%%%%%%%%%%%%%%%%
%%%%%%%%%%%%%%%%%%%%
\section{Preliminaries}
\label{sec.prelim}
%%%%%%%%%%%%%%%%%%%%
%%%%%%%%%%%%%%%%%%%%

%%%%%%%%%%
%%%%%%%%%%
\subsection{Operators}
\label{subsec.op}
%%%%%%%%%%
%%%%%%%%%%

This subsection collects the operators in the cylindrical coordinates on $\Omega$. Let $(r,\theta,z)$ be the coordinates defined in the introduction and ${\bf e}_r,{\bf e}_\theta,{\bf e}_z$ be the unit vectors.

Let $v$ be an axisymmetric vector field on $\Omega$, namely, 
$$
v(r,\theta,z)
= 
v^r(r,z) {\bf e}_r
+ v^\theta(r,z) {\bf e}_\theta
+ v^z(r,z) {\bf e}_z.
$$
Then, by the well-known formulas for $u = u^r {\bf e}_r + u^\theta {\bf e}_\theta + u^z {\bf e}_z$
\begin{align*}
\begin{split}
\opdiv u
&=
\dfrac{1}{r} \partial_r (ru^r)
+ \frac{1}{r} \partial_\theta u^\theta
+ \partial_z u^z, \\
\oprot u
&=
\Big(
\frac{1}{r} \partial_\theta u^z
- \partial_z u^\theta
\Big) {\bf e}_r 
+ (\partial_z u^r - \partial_r u^z) {\bf e}_\theta 
+ \dfrac{1}{r} 
\Big(
\partial_r (ru^\theta) 
- \partial_\theta u^r
\Big) {\bf e}_z, 
\end{split}
\end{align*}
we have 
\begin{align}\label{formula1}
\begin{split}
\opdiv v
&= 
\dfrac{1}{r} \partial_r (rv^r) + \partial_z v^z, \\
\oprot v 
&= 
- \partial_z v^\theta {\bf e}_r
+ (\partial_z v^r - \partial_r v^z) {\bf e}_\theta
+ \dfrac{1}{r} \partial_r (rv^\theta) {\bf e}_z.
\end{split}
\end{align}
The vector Laplacian on $v$ is represented as 
\begin{align}\label{formula2}
\begin{split}
\Big\{
\Big(\Delta - \frac{1}{r^2} \Big) v^r
\Big\}
{\bf e}_r 
+ \Big\{
\Big(\Delta - \frac{1}{r^2} \Big) v^\theta
\Big\}
{\bf e}_\theta
+ (\Delta v^z) 
{\bf e}_z, 
\end{split}
\end{align}
where $\Delta$ denotes the differential operator on axisymmetric scalar fields
\begin{align}\label{formula3}
\begin{split}
\Delta 
= \partial_r^2 + \frac{1}{r} \partial_r + \partial_z^2. 
\end{split}
\end{align}
Moreover, by the formula for $W=W^r {\bf e}_r + W^\theta {\bf e}_\theta$ 
and 
$\omega = \omega^r {\bf e}_r + \omega^\theta {\bf e}_\theta + \omega^z {\bf e}_z$
\begin{align}\label{formula4}
\begin{split}
W\times\omega 
&= 
(W^r {\bf e}_r + W^\theta {\bf e}_\theta) 
\times (\omega^r {\bf e}_r + \omega^\theta {\bf e}_\theta + \omega^z {\bf e}_z) \\
&= 
W^\theta \omega^z {\bf e}_r 
- W^r \omega^z {\bf e}_\theta 
+ (W^r \omega^\theta - W^\theta \omega^r) {\bf e}_z, 
\end{split}
\end{align}
we have, for $V$ in \eqref{def.V}, 
\begin{align}\label{formula5}
\begin{split}
V\times\oprot v
&=
- V^r (\oprot v)^z {\bf e}_\theta 
+ V^r (\oprot v)^\theta {\bf e}_z \\
&=
\frac{\gamma}{r^2} \partial_r (rv^\theta) {\bf e}_\theta
- \frac{\gamma}{r} (\partial_z v^r - \partial_r v^z) {\bf e}_z. 
\end{split}
\end{align}
%

%%%%%%%%%%
%%%%%%%%%%
\subsection{Function spaces}
\label{subsec.func.sp}
%%%%%%%%%%
%%%%%%%%%%

This subsection collects the definitions and properties of vector-valued function spaces used in this paper. The reader is referred to \cite[Section 2]{GigSaa2013} and \cite[Section 2]{GigSaa2015} for further details.

Let $E$ be a Banach space. We denote by $L^1(\R,E)$ the $E$-valued Lebesgue space with respect to the Lebesgue measure $\Lambda$ equipped with the norm $\|\cdot\|_{L^1(E)}$. Let $\mathcal{S}(\R)$ denote the space of rapidly decreasing functions. We extend the Fourier transformation on $\mathcal{S}(\R)$
$$
\hat{f}(\zeta) 
= \mathcal{F}f(\zeta) 
= \frac{1}{(2\pi)^{\frac12}} \int_{-\infty}^{\infty} e^{-\iu z\zeta} f(z) \dd z 
$$
by duality to the space of tempered distributions $\mathcal{S}'(\R,E)$, the set of bounded linear operators from $\mathcal{S}(\R)$ to $E$. This extension is again denoted by $\hat{f}$ or $\mathcal{F}f$ for $f\in\mathcal{S}'(\R,E)$. We also extend the inverse Fourier transformation on $\mathcal{S}(\R)$ to the space $\mathcal{S}'(\R,E)$ in the same manner and denote the extension by $\check{f}$ or $\mathcal{F}^{-1}f$ for $f\in\mathcal{S}'(\R,E)$.

Let $\mathcal{B}(\R)$ denote the Borel $\sigma$-algebra over $\R$ and $\mathcal{A}$ a $\sigma$-algebra over $\R$. A mapping $\mu:\mathcal{A}\to E$ is called a finite $E$-valued Radon measure if $\mu$ is an $E$-valued measure and the variation $|\mu|:\mathcal{A}\to[0,\infty)$ is a finite Radon measure, that is, a finite Borel regular measure. The set of all finite $E$-valued Radon measures is denoted by $\opM(\R,E)$. The space $\opM(\R,E)$ is a Banach space under the total variation norm $\|\cdot\|_{\opM(E)}=|\cdot|(\R)$.

Let $\rho\ge0$ and $\opX(\R,L^\infty_{\rho})$ be the space in \eqref{intro.eq.X} with the norm \eqref{intro.eq.X.norm} in the introduction. It follows by definition that $\opX(\R,L^\infty_{\rho})$ is a closed subspace of $\opM(\R,L^\infty_{\rho})$. We define a multiplication on $\opX(\R,L^\infty_{\rho})$ as follows. For given $f^{(1)}, f^{(2)}\in\opX(\R,L^\infty_{\rho})$ written as
$$
f^{(j)}=
g^{(j)} \dd \Lambda
+ \sum_{m\in\Z} a^{(j)}_m \delta_m, 
\quad
j=1,2, 
$$
we define the convolution 
\begin{align}\label{def.conv.X}
\begin{split}
f^{(1)}\ast f^{(2)}
&=
\bigg(
g^{(1)}\ast g^{(2)} 
+ \sum_{m\in\Z} 
(
a^{(1)}_m \tau_m g^{(2)} 
+ a^{(2)}_m \tau_m g^{(1)}
)
\bigg) \dd \Lambda \\
&\quad
+ \sum_{m\in\Z} 
\bigg(
\sum_{m=k+l} a^{(1)}_k a^{(2)}_l
\bigg)
\delta_m, 
\end{split}
\end{align}
where $\tau_m$ is the translation 
$
\tau_m f(r,\zeta)=f(r,\zeta+m)
$. 
Then it is easy to check that 
\begin{align}\label{est1.conv.X}
\begin{split}
\|f^{(1)}\ast f^{(2)}\|_{\opX(L^\infty_{\rho})}
\le 
\|f^{(1)}\|_{\opX(L^\infty_{\rho})}
\|f^{(2)}\|_{\opX(L^\infty_{\rho})}
\end{split}
\end{align}
by using the Young inequality. Thus the convolution \eqref{def.conv.X} defines a multiplication on $\opX(\R,L^\infty_{\rho})$. Moreover, if additionally $f^{(1)}\in \opX(\R,L^\infty_{\rho_1})$ and $f^{(2)}\in \opX(\R,L^\infty_{\rho_2})$, we have
\begin{align}\label{est2.conv.X}
\begin{split}
\|f^{(1)}\ast f^{(2)}\|_{\opX(L^\infty_{\rho_1+\rho_2})}
\le 
\|f^{(1)}\|_{\opX(L^\infty_{\rho_1})}
\|f^{(2)}\|_{\opX(L^\infty_{\rho_2})}. 
\end{split}
\end{align}

By embedding $\opX(\R,L^\infty_{\rho}) \hookrightarrow \mathcal{S}'(\R,L^\infty_{\rho})$, one can define the Fourier transformation
$$
\hat{f}(\zeta) 
= \mathcal{F}f(\zeta) 
= \widehat{g}(r,\zeta)
+ \sum_{m\in\Z} a_m(r) e^{\iu m \zeta}, 
\quad
f\in\opX(\R,L^\infty_{\rho}). 
$$
Therefore, the space $\opFX(\R,L^\infty_{\rho})$ in \eqref{intro.eq.FX} is well-defined with the norm $\|\hat{f}\|_{\opFX(L^\infty_{\rho})}=\|f\|_{\opX(L^\infty_{\rho})}$ and is continuously embedded in $\opBUC(\R,L^\infty_{\rho})$. By definition, the value of this norm does not change if the Fourier transformation is replaced by its inverse: 
$$
\|\cdot\|_{\opFX(L^\infty_{\rho})}
=\|\mathcal{F}^{-1}\cdot\|_{\opX(L^\infty_{\rho})}
=\|\mathcal{F}\cdot\|_{\opX(L^\infty_{\rho})}. 
$$
Also, the multiplication is well-defined. Indeed, for $\widehat{f^{(1)}},\widehat{f^{(2)}}\in \opFX(\R,L^\infty_{\rho})$, from
\begin{align*}
\begin{split}
\widehat{f^{(1)}} \widehat{f^{(2)}}
=
\widehat{f^{(1)}\ast f^{(2)}}
\end{split}
\end{align*}
and \eqref{est1.conv.X}, we have
\begin{align*}
\begin{split}
\|\widehat{f^{(1)}} \widehat{f^{(2)}}\|_{\opFX(L^\infty_{\rho})}
\le
\|\widehat{f^{(1)}}\|_{\opFX(L^\infty_{\rho})}
\|\widehat{f^{(2)}}\|_{\opFX(L^\infty_{\rho})}. 
\end{split}
\end{align*}
Moreover, if additionally $\widehat{f^{(1)}}\in \opFX(\R,L^\infty_{\rho_1})$ and $\widehat{f^{(2)}}\in \opFX(\R,L^\infty_{\rho_2})$, by \eqref{est2.conv.X}, we have 
\begin{align*}
\begin{split}
\|\widehat{f^{(1)}} \widehat{f^{(2)}}\|_{\opFX(L^\infty_{\rho_1+\rho_2})}
\le
\|\widehat{f^{(1)}}\|_{\opFX(L^\infty_{\rho_1})}
\|\widehat{f^{(2)}}\|_{\opFX(L^\infty_{\rho_2})} 
\end{split}
\end{align*}

The argument above shows the following lemma. 
%
%%%%%%%%%%
\begin{lemma}\label{lem.X.FX}
Let $\rho,\rho_1,\rho_2\ge0$. Then the following hold.
\begin{enumerate}[(1)]
\item\label{item1.lem.X.FX}
$\opX(\R,L^\infty_{\rho})$ is a Banach algebra under the convolution \eqref{def.conv.X}. Moreover, 
$$
\opX(\R,L^\infty_{\rho_1})\ast \opX(\R,L^\infty_{\rho_2})
\hookrightarrow
\opX(\R,L^\infty_{\rho_1+\rho_2}). 
$$

\item\label{item2.lem.X.FX}
$\opFX(\R,L^\infty_{\rho})$ is a Banach algebra under the multiplication and satisfies
$$
\opFX(\R,L^\infty_{\rho})\hookrightarrow\opBUC(\R,L^\infty_{\rho}). 
$$
Moreover, 
$$
\opFX(\R,L^\infty_{\rho_1})\cdot \opFX(\R,L^\infty_{\rho_2})
\hookrightarrow
\opFX(\R,L^\infty_{\rho_1+\rho_2}). 
$$
\end{enumerate}
\end{lemma}
%%%%%%%%%%
%
%
%%%%%%%%%%
\begin{remark}\label{rem.lem.X.FX}
In \cite{GigSaa2013,GigSaa2015}, the convolution and the Fourier transformation are defined on $\opM(\R,E)$ for spaces $E$ having the Radon-Nikod\'{y}m property \cite[Chapter $\mathrm{I}\hspace{-1.2pt}\mathrm{I}\hspace{-1.2pt}\mathrm{I}$, Definition 3]{DieUhlbook1977}, which is known to be satisfied by reflexive spaces $E$; see \cite[Chapter $\mathrm{I}\hspace{-1.2pt}\mathrm{I}\hspace{-1.2pt}\mathrm{I}$, Corollary 13]{DieUhlbook1977}. However, it is unclear to the author whether the weighted space $L^\infty_{\rho}$ in \eqref{intro.def.L} has this property. This is a reason why in this paper we are working in $\opX(\R,L^\infty_{\rho})$ on which both the convolution and the Fourier transformation can be defined in the explicit manners as above. 
\end{remark}
%%%%%%%%%%

%%%%%%%%%%%%%%%%%%%%
%%%%%%%%%%%%%%%%%%%%
\section{Axisymmetric Biot-Savart law}
\label{sec.ABSL}
%%%%%%%%%%%%%%%%%%%%
%%%%%%%%%%%%%%%%%%%%

In this section, we define the axisymmetric Biot-Savart law. Its applications given in Subsection \ref{subsec.app} are essential when studying the linearized problem of \eqref{intro.eq.NP} in the introduction.

For a given scalar field $\omega=\omega(r,z)$, we consider the problem
\begin{equation*}
\left\{
\begin{array}{ll}
-\Big(
\partial_r^2 
- \dfrac{1}{r} \partial_r 
+ \partial_z^2 
\Big) \widetilde{\psi} 
= 
r\omega
&\mbox{in}\ (1,\infty)\times\R \\ [10pt] 
\widetilde{\psi}=0
&\mbox{on}\ \{1\}\times\R. 
\end{array}\right.
\end{equation*}
The solution $\widetilde{\psi}$ is called the axisymmetric streamfunction or the Stokes streamfunction. In this paper, it is more convenient to use the functions $\psi=\psi(r,z)$ defined by
$$
\psi(r,z) = \frac{\widetilde{\psi}(r,z)}{r}. 
$$
It is easy to see that $\psi$ solves 
\begin{equation}\label{eq.BSL}
\left\{
\begin{array}{ll}
-\Big(
\Delta - \dfrac{1}{r^2}
\Big) \psi 
= 
\omega
&\mbox{in}\ (1,\infty)\times\R \\ [10pt] 
\psi=0
&\mbox{on}\ \{1\}\times\R. 
\end{array}\right.
\end{equation}
Using this solution $\psi$, we formally define the vector field $\mathcal{V} = \mathcal{V}[\omega]$ by 
\begin{align}\label{def.ABSL}
\begin{split}
&\mathcal{V}
= \mathcal{V}^r(r,z) {\bf e}_r  
+ \mathcal{V}^z(r,z) {\bf e}_z, \\
&\mathcal{V}^r(r,z)
:= -\partial_z\psi(r,z), 
\qquad 
\mathcal{V}^z(r,z)
:= \frac{1}{r} \partial_r(r\psi(r,z)). 
\end{split}
\end{align}
The formula $\mathcal{V}$ in \eqref{def.ABSL} is called the axisymmetric Biot-Savart law. It is clear that 
\begin{align}\label{prpty.BSL}
\begin{split}
&\opdiv \mathcal{V} = 0, 
\qquad
\oprot \mathcal{V}
= \omega {\bf e}_\theta, 
\qquad 
\mathcal{V}^r|_{\{1\}\times\R} = 0. 
\end{split}
\end{align}

Now we let $\rho\ge0$ and $\omega\in \opFX(\R,L^\infty_{\rho})$ and will solve \eqref{eq.BSL} using the Fourier transformation. The transform $\hat{\psi}$ solves the boundary value problem 
\begin{align}\label{eq.ode.strmfunc}
-\frac{\dd^2 \hat{\psi}}{\dd r^2} 
- \frac{1}{r} \frac{\dd \hat{\psi}}{\dd r} 
+ \Big(\zeta^2 + \frac{1}{r^2}\Big) \hat{\psi} 
= \hat{\omega}, \quad r>1, 
\qquad
\hat{\psi}(1) = 0. 
\end{align}
The linearly independent solutions of the homogeneous equation are
$$
\left\{
\begin{array}{ll}
K_{1}(|\zeta|r)
\mkern9mu \text{and} \mkern9mu 
I_{1}(|\zeta|r)
&\mbox{if}\ \zeta\neq0 \\ [5pt]
r^{-1}
\mkern9mu \text{and} \mkern9mu 
r 
&\mbox{if}\ \zeta=0, 
\end{array}
\right. 
$$
where $K_\nu(z), I_\nu(z)$ are the modified Bessel functions in Appendix \ref{app.lem}, with the Wronskian 
$$
\left\{
\begin{array}{ll}
r^{-1}
&\mbox{if}\ \zeta\neq0 \\ [5pt]
2r^{-1}
&\mbox{if}\ \zeta=0. 
\end{array}
\right. 
$$
Thus, if the Fourier transform of $\omega$ is given by 
$$
\hat{\omega} 
= h(r,\zeta) \dd \Lambda 
+ \sum_{m\in\Z} b_m(r) \delta_m, 
$$
the solution $\hat{\psi}$ of \eqref{eq.ode.strmfunc} belonging to $\bigcup_{\rho\ge0} \opX(\R,L^\infty_{\rho})$ has the form 
\begin{align}\label{rep.psi1}
\begin{split}
\hat{\psi}(r) 
=
\int_1^\infty 
(\hat{\omega}\lfloor \sigma_1)(r,s) \dd s 
+ \int_1^r 
(\hat{\omega}\lfloor \sigma_2)(r,s) \dd s 
+ \int_r^\infty 
(\hat{\omega}\lfloor \sigma_3)(r,s) \dd s. 
\end{split}
\end{align}
Here $\hat{\omega}\lfloor \sigma_i$ denotes the Borel measure 
\begin{align}\label{rep.psi2}
\begin{split}
\hat{\omega}\lfloor \sigma_i(r,s,\mathcal{O})
&=
\int_{\mathcal{O}}
\sigma_i(r,s,\zeta) h(s,\zeta)
\dd \Lambda \\
&\quad 
+ \sum_{m\in\Z\cap\mathcal{O}}
\sigma_i(r,s,m) b_m(s),
\quad \mathcal{O}\in \mathcal{B}(\R) 
\end{split}
\end{align}
and $\sigma_1, \sigma_2, \sigma_3$ are the functions 
\begin{align}\label{rep.psi3}
\begin{split}
\sigma_1(r,s,\zeta) 
&= 
\left\{
\begin{array}{ll}
-\dfrac{I_1(|\zeta|)}{K_1(|\zeta|)} K_1(|\zeta|r) sK_1(|\zeta|s) 
&\mbox{if}\ \zeta\neq0 \\ [10pt]
\dfrac12r^{-1}
&\mbox{if}\ \zeta=0
\end{array}
\right. \\
\sigma_2(r,s,\zeta) 
&= 
\left\{
\begin{array}{ll}
K_1(|\zeta|r) sI_1(|\zeta|s)
&\mbox{if}\ \zeta\neq0 \\ [10pt]
\dfrac12r^{-1}s^{2}
&\mbox{if}\ \zeta=0
\end{array}
\right. \\
\sigma_3(r,s,\zeta) 
&= 
\left\{
\begin{array}{ll}
I_1(|\zeta|r) sK_1(|\zeta|s)
&\mbox{if}\ \zeta\neq0 \\ [10pt]
\dfrac12r
&\mbox{if}\ \zeta=0. 
\end{array}
\right. 
\end{split}
\end{align}
A representation of $\psi$ is obtained by the inverse Fourier transform of $\hat{\psi}$ in \eqref{rep.psi1}--\eqref{rep.psi3}.

The solution $\hat{\psi}$ in \eqref{rep.psi1}--\eqref{rep.psi3} of the problem \eqref{eq.ode.strmfunc} is unique in $\bigcup_{\rho\ge0} \opX(\R,L^\infty_{\rho})$; see for instance \cite[Chapter $\mathrm{I}\hspace{-1.2pt}\mathrm{I}\hspace{-1.2pt}\mathrm{I}$]{Horbook2003} for the existence and uniqueness of solutions to the linear ordinary differential equations with smooth coefficients in the space of distributions.

%%%%%%%%%%
%%%%%%%%%%
\subsection{Well-definedness}
\label{subsec.WD}
%%%%%%%%%%
%%%%%%%%%%

In this subsection, by estimating $\hat{\psi}$ in \eqref{rep.psi1}--\eqref{rep.psi3}, we prove the well-definedness of the axisymmetric Biot-Savart law in \eqref{def.ABSL} for given $\omega \in \opFX(\R,L^\infty_{\rho})$ with $2<\rho<3$.

Estimates of $\hat{\psi}$ are given in Lemmas \ref{lem.strmfunc.L1}--\ref{lem.strmfunc.seq} below. 
%
%%%%%%%%%%
\begin{lemma}\label{lem.strmfunc.L1}
Let $2<\rho<3$ and $\omega\in \opFX(\R,L^\infty_{\rho})$ be given by, for some $h\in L^1(\R,L^\infty_{\rho})$, 
$$
\hat{\omega}=h(r,\zeta) \dd \Lambda. 
$$
The solution $\hat{\psi}$ of \eqref{eq.ode.strmfunc} has the form 
\begin{align}\label{eq1.lem.strmfunc.L1}
\begin{split}
\hat{\psi} = \hat{\psi}[h](r,\zeta) \dd \Lambda 
\end{split}
\end{align}
with
\begin{align}\label{eq2.lem.strmfunc.L1}
\begin{split}
\hat{\psi}[h](r,\zeta)
&=
-d[h](\zeta) K_1(|\zeta|r) 
+ K_1(|\zeta|r) 
\int_1^r I_1(|\zeta|s) s h(s,\zeta) \dd s \\ 
&\quad 
+ I_1(|\zeta|r) 
\int_r^\infty K_1(|\zeta|s) s h(s,\zeta) \dd s, 
\quad \zeta\neq0, \\
d[h](\zeta) 
&:=
\frac{I_{1}(|\zeta|)}{K_{1}(|\zeta|)}
\int_{1}^{\infty}
K_{1}(|\zeta|s) s h(s,\zeta) \dd s, 
\quad \zeta\neq0. 
\end{split}
\end{align}
Moreover, $\hat{\psi}$ satisfies the estimate
\begin{align}\label{est.lem.strmfunc.L1}
\begin{split}
\Big\|
\zeta
\mapsto
\Big(\frac{\dd}{\dd r}, \iu \zeta\Big)^{j} 
\hat{\psi}[h](r,\zeta)
\Big\|_{L^1(L^\infty_{\rho-2+j})}
\lesssim
\|h\|_{L^1(L^\infty_{\rho})}, 
\quad j=0,1,2 
\end{split}
\end{align}
and the relation
\begin{align}\label{eq3.lem.strmfunc.L1}
\begin{split}
\frac{\dd}{\dd r} \hat{\psi}[h](1,\zeta)
= \frac{d[h](\zeta)}{I_1(|\zeta|)}, 
\quad \zeta\neq0. 
\end{split}
\end{align}
The implicit constant depends on $\rho$.
\end{lemma}
%%%%%%%%%%
%
%
%%%%%
\begin{proof}
We may identify the measure $h\dd \Lambda$ with the function $h$ and focus on the proof of \eqref{est.lem.strmfunc.L1} and \eqref{eq3.lem.strmfunc.L1}. Indeed, the Fubini theorem is applicable in \eqref{rep.psi1}--\eqref{rep.psi2} thanks to the estimates \eqref{est6.proof.lem.strmfunc.L1}--\eqref{est9.proof.lem.strmfunc.L1} below. Thus the formula \eqref{rep.psi1}--\eqref{rep.psi3} can be written as \eqref{eq1.lem.strmfunc.L1}--\eqref{eq2.lem.strmfunc.L1}.

In the following proof, Lemma \ref{lem.MBF} in Appendix \ref{app.lem} will be used without any reference. Let us estimate $\sigma_i(r,s,\zeta) h(s,\zeta)$ in \eqref{rep.psi2} for $i=1,2,3$. When $0<|\zeta|\le1$, we have 
\begin{align}\label{est1.proof.lem.strmfunc.L1}
\begin{split}
&|\sigma_1(r,s,\zeta) h(s,\zeta)| 
=
\bigg|
-\dfrac{I_1(|\zeta|)}{K_1(|\zeta|)} K_1(|\zeta|r) 
s K_1(|\zeta|s) h(s,\zeta)
\bigg| \\
&\lesssim
\left\{
\begin{array}{ll}
|\zeta| K_1(|\zeta|r) |h(s,\zeta)|
&\mbox{if}\ s \le |\zeta|^{-1} \\ [5pt]
|\zeta|^{\frac32} K_1(|\zeta|r)
s^{\frac12} e^{-|\zeta|s}
|h(s,\zeta)| 
\quad 
&\mbox{if}\ s \ge |\zeta|^{-1} 
\end{array}\right. \\
&\lesssim
\left\{
\begin{array}{ll}
|\zeta|
K_1(|\zeta|r) s^{-\rho}
\|h(\zeta)\|_{L^\infty_\rho}
&\mbox{if}\ s \le |\zeta|^{-1} \\ [5pt]
|\zeta|^{\frac32}
K_1(|\zeta|r)
s^{-\rho+\frac12} e^{-|\zeta|s}
\|h(\zeta)\|_{L^\infty_\rho} 
\quad 
&\mbox{if}\ s \ge |\zeta|^{-1}. 
\end{array}\right.
\end{split}
\end{align}
For $i=2,3$, we have 
\begin{align}\label{est2.proof.lem.strmfunc.L1}
\begin{split}
&|\sigma_2(r,s,\zeta) h(s,\zeta)| 
= 
|K_1(|\zeta|r) sI_1(|\zeta|s) h(s,\zeta)| \\
&\lesssim
\left\{
\begin{array}{ll}
r^{-1} s^{2}
|h(s,\zeta)|
&\mbox{if}\ s \le r \le |\zeta|^{-1} \\ [5pt]
|\zeta|^{\frac12} r^{-\frac12} e^{-|\zeta|r} s^{2} 
|h(s,\zeta)|
&\mbox{if}\ s \le |\zeta|^{-1} \le r \\ [5pt]
|\zeta|^{-1} r^{-\frac12} s^{\frac12} e^{-|\zeta|(r-s)}
|h(s,\zeta)|
&\mbox{if}\ |\zeta|^{-1} \le s \le r 
\end{array}\right. \\
&\lesssim
\left\{
\begin{array}{ll}
r^{-1} s^{-\rho+2}
\|h(\zeta)\|_{L^\infty_\rho}
&\mbox{if}\ s \le r \le |\zeta|^{-1} \\ [5pt]
|\zeta| e^{-|\zeta|r} s^{-\rho+2} 
\|h(\zeta)\|_{L^\infty_\rho}
&\mbox{if}\ s \le |\zeta|^{-1} \le r \\ [5pt]
s^{-\rho+1} e^{-|\zeta|(r-s)}
\|h(\zeta)\|_{L^\infty_\rho}
&\mbox{if}\ |\zeta|^{-1} \le s \le r 
\end{array}\right.
\end{split}
\end{align}
and 
\begin{align}\label{est3.proof.lem.strmfunc.L1}
\begin{split}
&|\sigma_3(r,s,\zeta) h(s,\zeta)| = |I_1(|\zeta|r) K_1(|\zeta|s) s h(s,\zeta)| \\
&\lesssim
\left\{
\begin{array}{ll}
r |h(s,\zeta)|
&\mbox{if}\ r \le s \le |\zeta|^{-1} \\ [5pt]
|\zeta|^{\frac12} r s^{\frac12}
e^{-|\zeta|s} 
|h(s,\zeta)|
&\mbox{if}\ r \le |\zeta|^{-1} \le s \\ [5pt]
|\zeta|^{-1} r^{-\frac12} 
s^{\frac12} e^{|\zeta|(r-s)}
|h(s,\zeta)|
&\mbox{if}\ |\zeta|^{-1} \le r \le s 
\end{array}\right. \\
&\lesssim
\left\{
\begin{array}{ll}
r s^{-\rho}
\|h(\zeta)\|_{L^\infty_\rho}
&\mbox{if}\ r \le s \le |\zeta|^{-1} \\ [5pt]
|\zeta|^{-\frac12} s^{-\rho+\frac12} e^{-|\zeta|s} 
\|h(\zeta)\|_{L^\infty_\rho} 
&\mbox{if}\ r \le |\zeta|^{-1} \le s \\ [5pt]
|\zeta|^{-1} r^{-\frac12} 
s^{-\rho+\frac12} e^{|\zeta|(r-s)}
\|h(\zeta)\|_{L^\infty_\rho}
&\mbox{if}\ |\zeta|^{-1} \le r \le s. 
\end{array}\right.
\end{split}
\end{align}
When $|\zeta|\ge 1$, we have 
\begin{align}\label{est4.proof.lem.strmfunc.L1}
\begin{split}
|\sigma_1(r,s,\zeta) h(s,\zeta)|
&\lesssim
|\zeta|^{-\frac12}
e^{2|\zeta|}
K_1(|\zeta|r)
s^{\frac12} e^{-|\zeta|s}
|h(s,\zeta)| \\
&\lesssim
|\zeta|^{-\frac12}
e^{2|\zeta|}
K_1(|\zeta|r)
s^{-\rho+\frac12} e^{-|\zeta|s}
\|h(\zeta)\|_{L^\infty_\rho}. 
\end{split}
\end{align}
For $i=2,3$, we have 
\begin{align}\label{est5.proof.lem.strmfunc.L1}
\begin{split}
|\sigma_2(r,s,\zeta) h(s,\zeta)|
&\lesssim 
|\zeta|^{-1}
r^{-\frac12} s^{\frac12} e^{-|\zeta|(r-s)}
|h(s,\zeta)| \\
&\lesssim
|\zeta|^{-1}
r^{-\frac12} s^{-\rho+\frac12} e^{-|\zeta|(r-s)}
\|h(\zeta)\|_{L^\infty_\rho},
\quad
s \le r 
\end{split}
\end{align}
and
\begin{align}\label{est6.proof.lem.strmfunc.L1}
\begin{split}
|\sigma_3(r,s,\zeta) h(s,\zeta)|
&\lesssim
|\zeta|^{-1}
r^{-\frac12} s^{\frac12} e^{|\zeta|(r-s)}
|h(s,\zeta)| \\
&\lesssim
|\zeta|^{-1}
r^{-\frac12} s^{-\rho+\frac12} e^{|\zeta|(r-s)}
\|h(\zeta)\|_{L^\infty_\rho}, 
\quad
r \le s. 
\end{split}
\end{align}
Then we see from \eqref{est1.proof.lem.strmfunc.L1} that, when $0<|\zeta|\le1$, 
\begin{align}\label{est7.proof.lem.strmfunc.L1}
\begin{split}
\int_{1}^{\infty}
|\sigma_1(r,s,\zeta) h(s,\zeta)|
\dd s 
&\lesssim
(|\zeta|+|\zeta|^{\rho})
K_1(|\zeta|r) 
\|h(\zeta)\|_{L^\infty_\rho} \\
&\lesssim
\left\{
\begin{array}{ll}
r^{-1} 
\|h(\zeta)\|_{L^\infty_\rho}
&\mbox{if}\ r \le |\zeta|^{-1} \\ [5pt]
|\zeta| e^{-|\zeta|r} 
\|h(\zeta)\|_{L^\infty_\rho}
&\mbox{if}\ r \ge |\zeta|^{-1} 
\end{array}\right.
\end{split}
\end{align}
and from \eqref{est4.proof.lem.strmfunc.L1} that, when $|\zeta|\ge1$, 
\begin{align}\label{est8.proof.lem.strmfunc.L1}
\begin{split}
\int_{1}^{\infty}
|\sigma_1(r,s,\zeta) h(s,\zeta)|
\dd s 
&\lesssim
|\zeta|^{-\frac32}
e^{|\zeta|}
K_1(|\zeta|r) 
\|h(\zeta)\|_{L^\infty_\rho} \\
&\lesssim
|\zeta|^{-2}
e^{-|\zeta|(r-1)}
\|h(\zeta)\|_{L^\infty_\rho}. 
\end{split}
\end{align}
For $i=2,3$, we see from \eqref{est2.proof.lem.strmfunc.L1}--\eqref{est3.proof.lem.strmfunc.L1} that, when $0<|\zeta|\le1$, 
\begin{align}\label{est9.proof.lem.strmfunc.L1}
\begin{split}
&\int_1^r
|\sigma_2(r,s,\zeta) h(s,\zeta)|
\dd s 
+ \int_{r}^{\infty}
|\sigma_3(r,s,\zeta) h(s,\zeta)|
\dd s \\
&\lesssim
\left\{
\begin{array}{ll}
(r^{-\rho+2} + |\zeta|^{\rho-2}) \|h(\zeta)\|_{L^\infty_\rho}
&\mbox{if}\ r \le |\zeta|^{-1} \\ [5pt]
(|\zeta|^{\rho-2} e^{-|\zeta|r} + |\zeta|^{-2} r^{-\rho})
\|h(\zeta)\|_{L^\infty_\rho}
&\mbox{if}\ r \ge |\zeta|^{-1} 
\end{array}\right.
\end{split}
\end{align}
and from \eqref{est5.proof.lem.strmfunc.L1}--\eqref{est6.proof.lem.strmfunc.L1} that, when $|\zeta|\ge1$, 
\begin{align}\label{est10.proof.lem.strmfunc.L1}
\begin{split}
&\int_{1}^{r}
|\sigma_2(r,s,\zeta) h(s,\zeta)|
\dd s
+ \int_{r}^{\infty}
|\sigma_3(r,s,\zeta) h(s,\zeta)|
\dd s \\
&\lesssim
|\zeta|^{-2} r^{-\rho}
\|h(\zeta)\|_{L^\infty_\rho}. 
\end{split}
\end{align}

Now the estimates \eqref{est7.proof.lem.strmfunc.L1}--\eqref{est10.proof.lem.strmfunc.L1} verify the formula \eqref{eq1.lem.strmfunc.L1}--\eqref{eq2.lem.strmfunc.L1} as discussed at the begging of the proof. Moreover, the assertion \eqref{est.lem.strmfunc.L1} follows from \eqref{est7.proof.lem.strmfunc.L1}--\eqref{est10.proof.lem.strmfunc.L1} when $j=0$, from \eqref{eq.rel} in Appendix \ref{app.lem} when $j=1$, and from the equation \eqref{eq.ode.strmfunc} when $j=2$.

It remains to prove \eqref{eq3.lem.strmfunc.L1}. By direct computation, we find 
\begin{align*}
\begin{split}
\frac{\dd}{\dd r} \hat{\psi}[h](1,\zeta)
&= 
-d[h](\zeta) 
|\zeta| \frac{\dd K_1}{\dd z}(|\zeta|) 
+ |\zeta| \frac{\dd I_1}{\dd z}(|\zeta|) 
\int_1^\infty 
K_1(|\zeta|s) s h(s) \dd s \\
&= 
\frac{d[h](\zeta)}{I_1(|\zeta|)}
|\zeta| \Big(K_1(|\zeta|) \frac{\dd I_1}{\dd z}(|\zeta|) - I_1(|\zeta|) \frac{\dd K_1}{\dd z}(|\zeta|)\Big), 
\quad \zeta\neq0. 
\end{split}
\end{align*}
This combined with \eqref{def.W} leads to \eqref{eq3.lem.strmfunc.L1}. The proof of Lemma \ref{lem.strmfunc.L1} is complete. 
\end{proof}
%%%%%
%

%
%%%%%%%%%%
\begin{lemma}\label{lem.strmfunc.seq}
Let $2<\rho<3$ and $\omega\in \opFX(\R,L^\infty_{\rho})$ be given by, for some $\{b_m\}\in l^1(\Z,L^\infty_{\rho})$, 
$$
\hat{\omega} = \sum_{m\in\Z} b_m(r) \delta_m. 
$$
The solution $\hat{\psi}$ of \eqref{eq.ode.strmfunc} has the form 
\begin{align}\label{eq1.lem.strmfunc.seq}
\begin{split}
\hat{\psi} = \sum_{m\in\Z} \hat{\psi}[b_m](r) \delta_m 
\end{split}
\end{align}
with 
\begin{align}\label{eq2.lem.strmfunc.seq}
\begin{split}
\hat{\psi}[b_0](r)
&=
-d[b_0] r^{-1}
+ \frac{r^{-1}}{2} \int_1^r s^2 b_0(s) \dd s 
+ \frac{r}{2} \int_r^\infty b_0(s) \dd s, \\
d[b_0] 
&:=
\frac12 \int_{1}^{\infty} b_0(s) \dd s 
\end{split}
\end{align}
and 
\begin{align}\label{eq2'.lem.strmfunc.seq}
\begin{split}
\hat{\psi}[b_m](r)
&=
-d[b_m] K_1(|m|r) 
+ K_1(|m|r) 
\int_1^r 
I_1(|m|s) s b_m(s) \dd s \\ 
&\quad 
+ I_1(|m|r) 
\int_r^\infty 
K_1(|m|s) s b_m(s) \dd s,
\quad m\neq0, \\
d[b_m] 
&:=
\frac{I_{1}(|m|)}{K_{1}(|m|)}
\int_{1}^{\infty}
K_{1}(|m|s) s b_m(s) \dd s, 
\quad m\neq0. 
\end{split}
\end{align}
Moreover, $\hat{\psi}$ satisfies the estimate
\begin{align}\label{est.lem.strmfunc.seq}
\begin{split}
\Big\|
m
\mapsto
\Big(\frac{\dd}{\dd r}, \iu m\Big)^{j}
\hat{\psi}[b_m](r)
\Big\|_{l^1(L^\infty_{\rho-2+j})}
\lesssim
\|\{b_m\}\|_{l^1(L^\infty_{\rho})}, 
\quad j=0,1,2 
\end{split}
\end{align}
and the relation
\begin{align}\label{eq3.lem.strmfunc.seq}
\begin{split}
\frac{\dd}{\dd r} \hat{\psi}[b_m](1)
= 
\left\{
\begin{array}{ll}
2d[b_0]
&\mbox{if}\ m=0 \\ [5pt]
\dfrac{d[b_m]}{I_1(|m|)} 
\quad 
&\mbox{if}\ m\neq0.  
\end{array}\right.
\end{split}
\end{align}
The implicit constant depends on $\rho$.
\end{lemma}
%%%%%%%%%%
%
%
%%%%%
\begin{proof}
The formula \eqref{eq1.lem.strmfunc.seq}--\eqref{eq2'.lem.strmfunc.seq} can be derived from \eqref{rep.psi1}--\eqref{rep.psi3}. Thus we focus on the proof of \eqref{est.lem.strmfunc.seq} and \eqref{eq3.lem.strmfunc.seq}. By a simple calculation 
$$
r^{-1} \int_{1}^{r} s^{-\rho+2} \dd s
+ r \int_{r}^{\infty} s^{-\rho} \dd s 
\lesssim
r^{-\rho+2}, 
$$
we have 
\begin{align}\label{est1.proof.lem.strmfunc.seq}
\begin{split}
\Big\|
\frac{\dd^j}{\dd r^j} \hat{\psi}[b_0]
\Big\|_{L^\infty_{\rho-2+j}}
\lesssim
\|b_0\|_{L^\infty_{\rho}}, 
\quad j=0,1,2. 
\end{split}
\end{align}
By adopting the proof of Lemma \ref{lem.strmfunc.L1}, we have 
\begin{align}\label{est2.proof.lem.strmfunc.seq}
\begin{split}
\Big\|
\Big(\frac{\dd}{\dd r}, \iu m\Big)^{j}
\hat{\psi}[b_m]
\Big\|_{L^\infty_{\rho-2+j}}
\lesssim
\|b_m\|_{L^\infty_{\rho}}, 
\quad m\neq0, 
\quad j=0,1,2, 
\end{split}
\end{align}
with an implicit constant independent of $m$. Then the estimate \eqref{est.lem.strmfunc.seq} is obtained from \eqref{est1.proof.lem.strmfunc.seq}--\eqref{est2.proof.lem.strmfunc.seq}. The relation \eqref{eq3.lem.strmfunc.seq} when $m=0$ follows from direct computation. The one when $m\neq0$ follows from the proof of \eqref{eq3.lem.strmfunc.L1} in Lemma \ref{lem.strmfunc.L1}. This completes the proof. 
\end{proof}
%%%%%
%

A corollary of Lemmas \ref{lem.strmfunc.L1}--\ref{lem.strmfunc.seq} is the well-definedness of the formula \eqref{def.ABSL}.
%
%%%%%%%%%%
\begin{corollary}\label{cor.est.BSL}
Let $2<\rho<3$ and $\omega \in \opFX(\R,L^\infty_{\rho})$. The vector field $\mathcal{V}$ in \eqref{def.ABSL} defined from the solution $\psi\in \opFX(\R,L^\infty_{\rho-2})$ of \eqref{eq.BSL} is well-defined. Moreover, $\mathcal{V}$ satisfies 
\begin{align}\label{est.cor.est.BSL}
\begin{split}
\|
(\partial_r,\partial_z)^{j} \mathcal{V}
\|_{\opFX(L^\infty_{\rho-1+j})}
\lesssim
\|\omega\|_{\opFX(L^\infty_{\rho})}, 
\quad j=0,1. 
\end{split}
\end{align}
The implicit constant depends on $\rho$. 
\end{corollary}
%%%%%%%%%%
%
%
%%%%%
\begin{proof}
Denote
$$
\hat{\omega}
= h(r,\zeta) \dd \Lambda + \sum_{m\in\Z} b_m(r) \delta_m. 
$$
Lemmas \ref{lem.strmfunc.L1}--\ref{lem.strmfunc.seq} show that the inverse Fourier transform of 
$$
\hat{\psi}
= \hat{\psi}[h](r,\zeta) \dd \Lambda + \sum_{m\in\Z} \hat{\psi}[b_m](r) \delta_m, 
$$
where $\hat{\psi}[h]$ and $\hat{\psi}[b_m]$ are the functions in the lemmas, is a unique solution of \eqref{eq.BSL} in $\opFX(\R,L^\infty_{\rho-2})$. Hence the vector field $\mathcal{V}$ in \eqref{def.ABSL}, that is, in 
\begin{align}\label{def'.BSL}
\begin{split}
&\mathcal{V}
= \mathcal{V}^r(r,z) {\bf e}_r  
+ \mathcal{V}^z(r,z) {\bf e}_z, \\
&\mathcal{V}^r(r,z)
:= \mathcal{F}^{-1}[\zeta \mapsto -\iu \zeta \hat{\psi}(r,\zeta)], \\
&\mathcal{V}^z(r,z)
:= \mathcal{F}^{-1}\Big[\zeta \mapsto \frac{1}{r} \frac{\dd}{\dd r}(r\hat{\psi}(r,\zeta))\Big] 
\end{split}
\end{align}
satisfies the estimate \eqref{est.cor.est.BSL} thanks to \eqref{est.lem.strmfunc.L1} and \eqref{est.lem.strmfunc.seq}. The proof is complete. 
\end{proof}
%%%%%
%

%%%%%%%%%%
%%%%%%%%%%
\subsection{Applications}
\label{subsec.app}
%%%%%%%%%%
%%%%%%%%%%

We give two applications of the axisymmetric Biot-Savart law \eqref{def.ABSL}. The first one concerns a representation of an axisymmetric vector field without swirl.
%
%%%%%%%%%%
\begin{proposition}\label{prop.BSL}
Let $2<\rho<3$ and $v\in \opFX(\R,L^\infty_{\rho-1})^3\cap W^{1,2}_{{\rm loc}}(\overline{\Omega})^3$ have no swirl. Suppose that $\omega := (\oprot v)^\theta\in \opFX(\R,L^\infty_{\rho})$. Then the following hold.
\begin{enumerate}[(1)]
\item\label{item1.lem.BSL}
If $\opdiv v=0$ and $v=0$ on $\{1\}\times\R$, then $v=\mathcal{V}[\omega]$.

\item\label{item2.lem.BSL}
If, in addition, $\hat{\omega}$ is given by 
$$
\hat{\omega} = h(r,\zeta) \dd\Lambda + \sum_{m\in\Z} b_m(r) \delta_m, 
$$
then 
$$
d[h]=0 
\qquad \text{and} \qquad
d[b_m]=0, 
\quad
m\in \Z. 
$$
\end{enumerate}
\end{proposition}
%%%%%%%%%%
%
%
%%%%%
\begin{proof}
(\ref{item1.lem.BSL})
We will show that $u:=v-\mathcal{V}[\omega]\in \opFX(\R,L^\infty_{\rho-1})^3$ is zero. By definition, 
\begin{align*}
\begin{split}
&\opdiv u = 0, 
\qquad
\oprot u = 0, 
\qquad 
u^r|_{\{1\}\times\R} = 0. 
\end{split}
\end{align*}
In particular, $\Delta u=0$ in the sense of distributions. From \eqref{formula2}, we see that $u^r$ solves 
\begin{equation*}
\left\{
\begin{array}{ll}
-\Big(
\Delta - \dfrac{1}{r^2}
\Big) u^r
= 
0
&\mbox{in}\ (1,\infty)\times\R \\ [10pt] 
u^r=0
&\mbox{on}\ \{1\}\times\R. 
\end{array}\right.
\end{equation*}
Moreover, the Fourier transform $\hat{u}^r$ solves \eqref{eq.ode.strmfunc} with 
$\hat{\omega}=0$. Thanks to $\hat{u}^r\in \opX(\R,L^\infty_{\rho-1})$, the formula \eqref{rep.psi1}--\eqref{rep.psi3} gives $\hat{u}^r=0$ and hence $u^r=0$. Now, from 
\begin{align*}
\begin{split}
\partial_z u^z = \opdiv u = 0, 
\qquad
\partial_r u^z = -\oprot u = 0, 
\end{split}
\end{align*}
we see that $u^z$ is a constant. By $u^z\in \opFX(\R,L^\infty_{\rho-1})$, we have $u^z=0$ and hence $u=0$.

(\ref{item2.lem.BSL})
Recall that $\mathcal{V}$ is defined from the inverse Fourier transform of 
$$
\hat{\psi}
= \hat{\psi}[h](r,\zeta) \dd \Lambda + \sum_{m\in\Z} \hat{\psi}[b_m](r) \delta_m 
$$
with the functions $\hat{\psi}[h]$ and $\hat{\psi}[b_m]$ in Lemmas \ref{lem.strmfunc.L1}--\ref{lem.strmfunc.seq}. Since $\mathcal{V}^z=v^z=0$ on $\{1\}\times\R$ because of (\ref{item1.lem.BSL}) above, the definition of $\mathcal{V}^z$ in \eqref{def.ABSL}, or \eqref{def'.BSL}, leads to 
\begin{align*}
\begin{split}
\frac{\dd \hat{\psi}}{\dd r}(1,\zeta)
= \frac{1}{r} \frac{\dd}{\dd r} (r\hat{\psi}(r,\zeta)) \Big|_{\{1\}\times\R}
= \mathcal{F}\big[\mathcal{V}^z|_{\{1\}\times\R}\big]
= 0.
\end{split}
\end{align*}
Hence the assertion follows from the relations \eqref{eq3.lem.strmfunc.L1} and \eqref{eq3.lem.strmfunc.seq} in Lemmas \ref{lem.strmfunc.L1}--\ref{lem.strmfunc.seq}. 
\end{proof}
%%%%%
%

The second one concerns the existence of potentials for a rotation-free vector field.
%
%%%%%%%%%%
\begin{proposition}\label{prop.pot}
Let $f\in L^2_{{\rm loc}}(\overline{\Omega})$ be axisymmetric and have no swirl. If $\oprot f=0$ in the sense of distributions, there is an axisymmetric $p\in W^{1,2}_{{\rm loc}}(\overline{\Omega})$ such that 
$$
f
=\nabla p
=\partial_r p {\bf e}_r + \partial_z p {\bf e}_z.
$$
\end{proposition}
%%%%%%%%%%
%
%
%%%%%
\begin{proof}
To apply \cite[Chapter $\mathrm{I}\hspace{-1.2pt}\mathrm{I}$, Lemmas 1.1.5 and 2.2.1]{Sohbook2013}, we will show that $\langle f, \varphi \rangle=0$ for all $\varphi\in C^\infty_{0,\sigma}(\Omega)$. Remark that, when $f$ is smooth, integration by parts gives
\begin{align}\label{eq1.proof.prop.pot}
\begin{split}
&\langle
\oprot f, \chi {\bf e}_\theta 
\rangle 
=
\langle
(\oprot f)^\theta, 
\chi
\rangle \\
&=
\int_{-\infty}^\infty 
\int_1^\infty 
(\partial_z f^r(r,z) - \partial_r f^z(r,z))
\overline{\underline{\chi}(r,z)}
r \dd r \dd z \\
&=
\int_{-\infty}^{\infty}
\int_{1}^{\infty}
\Big(
f^r(r,z) 
\overline{(-\partial_z \underline{\chi}(r,z))}
+ f^z(r,z) 
\overline{\frac{1}{r} \partial_r (r\underline{\chi}(r,z))}
\Big)
r \dd r \dd z 
\end{split}
\end{align}
for all $\chi\in C^\infty_0(\Omega)$ and $\underline{\chi}(r,z):=\int_{0}^{2\pi}\chi(r,\theta,z)\dd\theta$.

Take $\varphi\in C^\infty_{0,\sigma}(\Omega)$. Since $f$ is axisymmetric and has no swirl, the vector field 
$$
\widetilde{\varphi}
= \underline{\varphi^r}(r,z) {\bf e}_r + \underline{\varphi^z}(r,z) {\bf e}_z 
$$ 
satisfies 
\begin{align}\label{eq2.proof.prop.pot}
\langle f, \varphi \rangle
= \langle f, \widetilde{\varphi} \rangle. 
\end{align}
Set $\omega=(\oprot\widetilde{\varphi})^\theta$. Then the Biot-Savart law in Proposition \ref{prop.BSL} shows that 
\begin{align}\label{eq3.proof.prop.pot}
\widetilde{\varphi} = \mathcal{V}
\end{align}
with $\mathcal{V}=\mathcal{V}[\omega]$ in \eqref{def.ABSL} defined from $\psi$ solving the equation \eqref{eq.BSL}.

We claim that $\psi=\psi(r,z)$ is smooth and compactly supported in $(1,\infty)\times\R$. Indeed, $\psi$ is smooth thanks to elliptic regularity for \eqref{eq.BSL}. Set
\begin{align*}
\begin{split}
S_{l}=\{(r,z)\in(1,\infty)\times\R~|~r\ge l \mkern9mu \text{or} \mkern9mu |z|\ge l\}, 
\quad l\ge 0. 
\end{split}
\end{align*}
Then, by \eqref{def.ABSL} and \eqref{eq3.proof.prop.pot}, there is a sufficiently large $L>0$ such that 
\begin{align*}
\begin{split}
\partial_z\psi(r,z) 
=
\partial_r(r\psi(r,z))
=0, 
\quad 
(r,z)\in S_{L}. 
\end{split}
\end{align*}
Then we find that, for some constant $c>0$, 
\begin{align*}
\begin{split}
\psi(r,z) = \psi(r) = \frac{c}{r}, 
\quad 
(r,z)\in S_{L}. 
\end{split}
\end{align*}
Since $\psi(1,z)=0$, putting $r=1$ and $|z|=L$ yields that $c=0$. Consequently, we have 
\begin{align*}
\begin{split}
\psi(r,z) = 0, 
\quad 
(r,z)\in S_{L}. 
\end{split}
\end{align*}
Thus the claim follows. We may regard $\psi$ as a function belonging to $C^\infty_0(\Omega)$.

Now we have, by \eqref{eq2.proof.prop.pot}--\eqref{eq3.proof.prop.pot} and the condition that $\oprot f=0$ with \eqref{eq1.proof.prop.pot}, 
\begin{align*}
\begin{split}
\langle f, \varphi \rangle
&= \langle f, \mathcal{V} \rangle \\
&=
2\pi \int_{-\infty}^{\infty}
\int_{1}^{\infty}
\Big(
f^r(r,z) 
\overline{(-\partial_z \psi(r,z))}
+ f^z(r,z) 
\overline{\frac{1}{r} \partial_r (r\psi(r,z))}
\Big)
r \dd r \dd z \\
&= 
0. 
\end{split}
\end{align*}
Therefore, by \cite[Chapter $\mathrm{I}\hspace{-1.2pt}\mathrm{I}$, Lemmas 1.1.5 and 2.2.1]{Sohbook2013}, we have $f=\nabla p$ for some $p\in W^{1,2}_{{\rm loc}}(\overline{\Omega})$. Hence the assertion follows since $f$ is axisymmetric and has no swirl. 
\end{proof}
%%%%%
%

%%%%%%%%%%%%%%%%%%%%
%%%%%%%%%%%%%%%%%%%%
\section{Linearized problem}
\label{sec.LP}
%%%%%%%%%%%%%%%%%%%%
%%%%%%%%%%%%%%%%%%%%

In this section, we consider the linearized problem of \eqref{intro.eq.NP}
\begin{equation}\tag{LP}\label{eq.LP}
\left\{
\begin{array}{ll}
-\Big(
\Delta - \dfrac{1}{r^2} 
\Big) v^r 
+ \partial_r q = f^r
&\mbox{in}\ (1,\infty)\times\R \\ [10pt]
-\Big(\Delta
+ \dfrac{\gamma}{r} \partial_r
- \dfrac{1-\gamma}{r^2}
\Big) v^\theta 
= f^\theta
&\mbox{in}\ (1,\infty)\times\R \\ [10pt]
-\Delta v^z 
+ \dfrac{\gamma}{r} (\partial_z v^r - \partial_r v^z)
+ \partial_z q
=
f^z
&\mbox{in}\ (1,\infty)\times\R \\ [10pt]
\dfrac{1}{r} \partial_r (rv^r) 
+ \partial_z v^z 
= 0
&\mbox{in}\ (1,\infty)\times\R \\ [10pt]
(v^r,v^\theta,v^z)
= (0,0,0)
&\mbox{on}\ \{1\}\times\R \\ [5pt]
(v^r,v^\theta,v^z)(r,\zeta)
\to (0,0,0)
&\mbox{as}\ r\to\infty. 
\end{array}\right.
\end{equation}

For the system \eqref{eq.LP}, we have the following proposition. 
%
%%%%%%%%%%
\begin{proposition}\label{prop.LP}
Let $\gamma>2$ and $2<\rho<3$ satisfy $\rho\le\gamma$. Let $f\in \opFX(\R,L^\infty_{\rho+1})^3$. There is a unique weak solution $v\in \opFX(\R,L^\infty_{\rho-1})^3\cap W^{1,2}_{{\rm loc}}(\overline{\Omega})^3$ of \eqref{eq.LP} satisfying 
\begin{align}\label{est.prop.LP}
\|(\partial_r,\partial_z)^j v\|_{\opFX(L^\infty_{\rho-1+j})} 
\lesssim
\|f\|_{\opFX(L^\infty_{\rho+1})}. 
\quad j=0,1. 
\end{align}
The implicit constant depends on $\gamma, \rho$.
\end{proposition}
%%%%%%%%%%
%
Proposition \ref{prop.LP} is a corollary of Propositions \ref{prop.LP1} and \ref{prop.LP2} in Subsections \ref{subsec.rad.vert} and \ref{subsec.swirl} below, respectively, and linearity of the system. Hence we omit the proof for simplicity.

We will study \eqref{eq.LP} by decomposing it into the system for $(v^r,v^z)$
\begin{equation}\tag{LP1}\label{eq.LP1}
\left\{
\begin{array}{ll}
-\Big(
\Delta - \dfrac{1}{r^2} 
\Big) v^r 
+ \partial_r q = f^r
&\mbox{in}\ (1,\infty)\times\R \\ [10pt]
-\Delta v^z 
+ \dfrac{\gamma}{r} (\partial_z v^r - \partial_r v^z)
+ \partial_z q
=
f^z
&\mbox{in}\ (1,\infty)\times\R \\ [10pt]
\dfrac{1}{r} \partial_r (rv^r) 
+ \partial_z v^z 
= 0
&\mbox{in}\ (1,\infty)\times\R \\ [10pt]
(v^r,v^z)
= (0,0)
&\mbox{on}\ \{1\}\times\R \\ [5pt]
(v^r,v^z)(r,\zeta)
\to (0,0)
&\mbox{as}\ r\to\infty 
\end{array}\right.
\end{equation}
and the single equation for $v^\theta$
\begin{equation}\tag{LP2}\label{eq.LP2}
\left\{
\begin{array}{ll}
-\Big(\Delta
+ \dfrac{\gamma}{r} \partial_r
- \dfrac{1-\gamma}{r^2}
\Big) v^\theta 
= f^\theta
&\mbox{in}\ (1,\infty)\times\R \\ [10pt]
v^\theta = 0 
&\mbox{on}\ \{1\}\times\R \\ [5pt]
v^\theta(r,z)\to0
&\mbox{as}\ r\to\infty. 
\end{array}\right. 
\end{equation}
%

%%%%%%%%%%
%%%%%%%%%%
\subsection{Estimate of radial and vertical components}
\label{subsec.rad.vert}
%%%%%%%%%%
%%%%%%%%%%

We consider the system \eqref{eq.LP1} for the components $(v^r,v^z)$ in this subsection.
%
%%%%%%%%%%
\begin{proposition}\label{prop.LP1}
Let $\gamma>2$ and $2<\rho<3$ satisfy $\rho\le\gamma$. Let $f\in \opFX(\R,L^\infty_{\rho+1})^3$ have no swirl. There is a unique weak solution $v\in \opFX(\R,L^\infty_{\rho-1})^3\cap W^{1,2}_{{\rm loc}}(\overline{\Omega})^3$ of \eqref{eq.LP1} satisfying 
\begin{align}\label{est.prop.LP1}
\|(\partial_r,\partial_z)^j v\|_{\opFX(L^\infty_{\rho-1+j})} 
\lesssim
\|f\|_{\opFX(L^\infty_{\rho+1})}, 
\quad j=0,1. 
\end{align}
The implicit constant depends on $\gamma, \rho$.
\end{proposition}
%%%%%%%%%%
%

We prove this proposition by using the axisymmetric Biot-Savart law in Section \ref{sec.ABSL}. For this purpose, we first derive the equation for the vorticity of the vector field $v^r {\bf e}_r+v^z {\bf e}_z$.

%
%%%%%%%%%%
\begin{lemma}\label{lem.eq.vol}
Let $\gamma\in\R$ and let $f\in C^\infty(\Omega)^3$ be axisymmetric and have no swirl. Suppose that $(v,\nabla q)$ is a smooth solution of \eqref{eq.LP}. Then the component $\omega:=(\oprot v)^\theta$ solves 
\begin{align}\label{eq.lem.eq.vol}
\begin{split}
-\Big(
\Delta + \dfrac{\gamma}{r}\partial_r - \dfrac{1+\gamma}{r^2} 
\Big) \omega 
= (\oprot f)^\theta. 
\end{split}
\end{align}
\end{lemma}
%%%%%%%%%%
%
%
%%%%%
\begin{proof}
By the first two lines of \eqref{eq.LP1}, we see that $\omega=\partial_z v^r - \partial_r v^z$ solves
\begin{align}
\begin{split}
-\Big(\Delta - \frac{1}{r^2}\Big) \omega
- \gamma \partial_r \Big(\frac{\omega}{r}\Big) 
= \partial_z f^r - \partial_r f^z, 
\end{split}
\end{align}
which implies the assertion. 
\end{proof}
%%%%%
%

Let $(v^r,v^z)$ be a solution of \eqref{eq.LP1} and $v:=v^r {\bf e}_r+v^z {\bf e}_z$. We give a representation of $\omega:=(\oprot v)^\theta$. By Lemma \ref{lem.eq.vol}, the Fourier transform $\hat{\omega}$ solves the equation
\begin{align}\label{eq.ode.vol}
-\frac{\dd^2 \hat{\omega}}{\dd r^2} 
- \frac{1+\gamma}{r} \frac{\dd \hat{\omega}}{\dd r} 
+ \Big(\zeta^2 + \frac{1+\gamma}{r^2} \Big) \hat{\omega} 
= \widehat{(\oprot f)}^\theta, 
\quad r>1. 
\end{align}
The linearly independent solutions of the homogeneous equation are
$$
\left\{
\begin{array}{ll}
r^{-\frac{\gamma}{2}} K_{|\frac{\gamma}2+1|}(|\zeta|r)
\mkern9mu \text{and} \mkern9mu 
r^{-\frac{\gamma}{2}} I_{|\frac{\gamma}2+1|}(|\zeta|r)
&\mbox{if}\ \zeta\neq0 \\ [5pt]
r^{-\gamma-1}
\mkern9mu \text{and} \mkern9mu 
r 
&\mbox{if}\ \zeta=0, 
\end{array}
\right. 
$$
where $K_\nu(z), I_\nu(z)$ are the modified Bessel functions in Appendix \ref{app.lem}, with the Wronskian 
$$
\left\{
\begin{array}{ll}
r^{-\frac{\gamma}{2}-1}
&\mbox{if}\ \zeta\neq0 \\ [5pt]
(\gamma+2)r^{-\gamma-1}
&\mbox{if}\ \zeta=0. 
\end{array}
\right. 
$$
Thus, if $(\oprot f)^\theta\in \opFX(\R,L^\infty_{\rho})$ for some $\rho\ge0$ and the Fourier transform is given by 
$$
\widehat{(\oprot f)}^\theta
= \tilde{h}(r,\zeta) \dd \Lambda 
+ \sum_{m\in\Z} \tilde{b}_m(r) \delta_m, 
$$
then the solution of \eqref{eq.ode.vol} belonging to $\bigcup_{\rho\ge0} \opX(\R,L^\infty_{\rho})$ has the form 
\begin{align}\label{rep.omega1}
\begin{split}
\hat{\omega}(r)
&=
c\lfloor \sigma_4(r) + \Phi(r), \\
\hat{\Phi}(r)
&:=
\int_1^r 
\widehat{(\oprot f)}^\theta\lfloor \sigma_5(s,r) \dd s 
+ \int_r^\infty 
\widehat{(\oprot f)}^\theta\lfloor \sigma_6(s,r) \dd s.
\end{split}
\end{align}
Here $c\lfloor \sigma_4(r)$ and $\widehat{(\oprot f)}^\theta\lfloor \sigma_i$ denote the Borel measures
\begin{align}\label{rep.omega2}
\begin{split}
c\lfloor \sigma_4(r,\mathcal{O})
&=
\int_{\mathcal{O}}
\sigma_4(r,\zeta) c(\zeta)
\dd \Lambda 
+ \sum_{m\in\Z\cap\mathcal{O}}
\sigma_4(r,m) c_m \delta_m, 
\quad \mathcal{O}\in \mathcal{B}(\R), 
\end{split}
\end{align}
for some function $c(\zeta)$ and series $\{c_m\}_{m\in\Z}$, and
\begin{align}\label{rep.omega3}
\begin{split}
\widehat{(\oprot f)}^\theta\lfloor \sigma_i(r,s,\mathcal{O})
&=
\int_{\mathcal{O}}
\sigma_i(r,s,\zeta) \widetilde{h}(s,\zeta)
\dd \Lambda \\
&\quad
+ \sum_{m\in\Z\cap\mathcal{O}}
\sigma_i(r,s,m) \widetilde{b}_m(s) \delta_m, 
\quad \mathcal{O}\in \mathcal{B}(\R), 
\end{split}
\end{align}
respectively. Moreover, $\sigma_4, \sigma_5, \sigma_6$ are the functions 
\begin{align}\label{rep.omega4}
\begin{split}
\sigma_4(r,\zeta) 
&= 
\left\{
\begin{array}{ll}
r^{-\frac{\gamma}{2}} K_{|\frac{\gamma}2+1|}(|\zeta|r) 
&\mbox{if}\ \zeta\neq0\\ [5pt]
r^{-\gamma-1} 
&\mbox{if}\ \zeta=0\\ 
\end{array}
\right. \\
\sigma_5(r,s,\zeta) 
&= 
\left\{
\begin{array}{ll}
r^{-\frac{\gamma}{2}}
K_{|\frac{\gamma}2+1|}(|\zeta|r)
s^{\frac{\gamma}{2}+1}
I_{|\frac{\gamma}2+1|}(|\zeta|s)
&\mbox{if}\ \zeta\neq0\\ [5pt]
\dfrac{1}{\gamma+2}
r^{-\gamma-1}
s^{\gamma+2}
&\mbox{if}\ \zeta=0\\ [5pt]
\end{array}
\right. \\
\sigma_6(r,s,\zeta) 
&= 
\left\{
\begin{array}{ll}
r^{-\frac{\gamma}{2}}
I_{|\frac{\gamma}2+1|}(|\zeta|r)
s^{\frac{\gamma}{2}+1}
K_{|\frac{\gamma}2+1|}(|\zeta|s) 
&\mbox{if}\ \zeta\neq0\\ [5pt]
\dfrac{1}{\gamma+2}
r
&\mbox{if}\ \zeta=0. \\ [5pt]
\end{array}
\right. 
\end{split}
\end{align}
A representation of $\omega$ is obtained by the inverse Fourier transform of $\hat{\omega}$ in \eqref{rep.omega1}--\eqref{rep.omega4}.

The proof of Proposition \ref{prop.LP1} uses Lemma \ref{lem.vel.L1}--\ref{lem.uniqueness} below.

%
%%%%%%%%%%
\begin{lemma}\label{lem.vel.L1}
Let $\gamma>2$ and $2<\rho<3$ satisfy $\rho\le\gamma$. Let $f\in \opFX(\R,L^\infty_{\rho+1})^3$ have no swirl and be given by, for some $g\in L^1(\R,L^\infty_{\rho+1})^3$, 
$$
\hat{f} = g(r,\theta,\zeta) \dd \Lambda. 
$$
There is a weak solution $v\in \opFX(\R,L^\infty_{\rho-1})^3\cap W^{1,2}_{{\rm loc}}(\overline{\Omega})^3$ of \eqref{eq.LP1} having the form  
\begin{align}\label{eq.lem.vel.L1}
\begin{split}
\hat{v} = \hat{v}[g](r,\theta,\zeta) \dd \Lambda
\end{split}
\end{align}
and satisfying 
\begin{align}\label{est.lem.vel.L1}
\begin{split}
\|
(\partial_r,\partial_z)^{j} v
\|_{\opFX(L^\infty_{\rho-1+j})}
\lesssim
\|g\|_{L^1(L^\infty_{\rho+1})}, 
\quad j=0,1. 
\end{split}
\end{align}
The implicit constant depends on $\gamma,\rho$.
\end{lemma}
%%%%%%%%%%
%
%
%%%%%
\begin{proof}
As in the proof of Lemma \ref{lem.strmfunc.L1}, we may identify $g\dd \Lambda$ with $g$. We will find a solution $v$ by applying the axisymmetric Biot-Savart law \eqref{def.ABSL}. Set $\omega:=(\oprot v)^\theta$.

First we assume that $g\in C^{\infty}_0(\Omega)^3$ is axisymmetric and has no swirl. Since 
$$
\widehat{(\oprot f)}^\theta(s,\zeta) 
=\widehat{(\oprot g)}^\theta(s,\zeta) 
=\iu \zeta g^r(s,\zeta) - \frac{\dd g^z}{\dd s}(s,\zeta) 
$$
and $\hat{\omega}$ solves \eqref{eq.ode.vol}, the formula \eqref{rep.omega1}--\eqref{rep.omega4} shows that 
\begin{align}\label{rep1.proof.lem.vel.L1}
\hat{\omega}(r,\zeta)
&=
c[g](\zeta) r^{-\frac{\gamma}{2}} K_{\frac{\gamma}2+1}(|\zeta|r) 
+ \hat{\Phi}[g](r,\zeta). 
\end{align}
Here $c[g](\zeta)$ is given in \eqref{rep4.proof.lem.vel.L1} below and 
\begin{align*}
\begin{split}
&\hat{\Phi}[g](r,\zeta) \\
&=
\int_1^r 
r^{-\frac{\gamma}{2}}
K_{\frac{\gamma}2+1}(|\zeta|r)
s^{\frac{\gamma}2+1}
I_{\frac{\gamma}2+1}(|\zeta|s)
\Big(
\iu \zeta g^r(s,\zeta) - \frac{\dd g^z}{\dd s}(s,\zeta) 
\Big) \dd s \\
&\quad
+ \int_r^\infty 
r^{-\frac{\gamma}{2}}
I_{\frac{\gamma}2+1}(|\zeta|r)
s^{\frac{\gamma}2+1}
K_{\frac{\gamma}2+1}(|\zeta|s)
\Big(
\iu \zeta g^r(s,\zeta) - \frac{\dd g^z}{\dd s}(s,\zeta)
\Big) \dd s. 
\end{split}
\end{align*}
We rewrite $\hat{\Phi}[g]$ in \eqref{rep1.proof.lem.vel.L1}. By integration by parts, one has 
\begin{align}\label{rep2.proof.lem.vel.L1}
\begin{split}
\hat{\Phi}[g](r,\zeta) 
&=
\int_1^r 
r^{-\frac{\gamma}{2}} 
K_{\frac{\gamma}2+1}(|\zeta|r) 
M[g](s,\zeta)
\dd s \\
&\quad
+ \int_r^\infty 
r^{-\frac{\gamma}{2}}
I_{\frac{\gamma}2+1}(|\zeta|r) 
N[g](s,\zeta)
\dd s 
\end{split}
\end{align}
with
\begin{align*}
\begin{split}
M[g](s,\zeta)
&=
\iu\zeta
s^{\frac{\gamma}2+1}
I_{\frac{\gamma}2+1}(|\zeta|s)
g^r(s,\zeta)
+ \frac{\dd}{\dd s} 
(
s^{\frac{\gamma}2+1}
I_{\frac{\gamma}2+1}(|\zeta|s)
)
g^z(s,\zeta), \\
N[g](s,\zeta)
&= 
\iu\zeta
s^{\frac{\gamma}2+1}
K_{\frac{\gamma}2+1}(|\zeta|s)
g^r(s,\zeta)
+ \frac{\dd}{\dd s} 
(
s^{\frac{\gamma}2+1}
K_{\frac{\gamma}2+1}(|\zeta|s)
)
g^z(s,\zeta). 
\end{split}
\end{align*}
By the relation \eqref{eq.rel} in Appendix \ref{app.lem}, we find 
\begin{align*}
\begin{split}
\frac{\dd}{\dd s} 
(
s^{\frac{\gamma}2+1}
I_{\frac{\gamma}2+1}(|\zeta|s)
)
&=
|\zeta| s^{\frac{\gamma}2+1}
I_{\frac{\gamma}2}(|\zeta|s), \\
\frac{\dd}{\dd s} 
(
s^{\frac{\gamma}2+1}
K_{\frac{\gamma}2+1}(|\zeta|s)
)
&= 
-|\zeta| s^{\frac{\gamma}2+1} 
K_{\frac{\gamma}2}(|\zeta|s). 
\end{split}
\end{align*}
Thus $M[g]$ and $N[g]$ can be rewritten as
\begin{align}\label{rep3.proof.lem.vel.L1}
\begin{split}
M[g](s,\zeta)
&=
\iu\zeta
s^{\frac{\gamma}2+1}
I_{\frac{\gamma}2+1}(|\zeta|s)
g^r(s,\zeta)
+ 
|\zeta| s^{\frac{\gamma}2+1}
I_{\frac{\gamma}2}(|\zeta|s)
g^z(s,\zeta), \\
N[g](s,\zeta)
&= 
\iu\zeta
s^{\frac{\gamma}2+1}
K_{\frac{\gamma}2+1}(|\zeta|s)
g^r(s,\zeta)
-|\zeta| s^{\frac{\gamma}2+1} 
K_{\frac{\gamma}2}(|\zeta|s)
g^z(s,\zeta). 
\end{split}
\end{align}
The function $c[g](\zeta)$ is chosen so that $d[h](\zeta)$ in Lemma \ref{lem.strmfunc.L1} with $h=\hat{\omega}$ is zero, namely, 
\begin{align*}
\begin{split}
\int_1^\infty 
s K_1(|\zeta|s) 
\Big(
c[g](\zeta) s^{-\frac{\gamma}{2}}
K_{\frac{\gamma}2+1}(|\zeta|s)
+ \hat{\Phi}[g](s,\zeta)
\Big) \dd s
=0, 
\quad 
\zeta\neq0, 
\end{split}
\end{align*}
which leads to 
\begin{align}\label{rep4.proof.lem.vel.L1}
\begin{split}
c[g](\zeta)
= -\frac{1}{F(\zeta)}
\int_1^\infty
s K_1(|\zeta|s) \hat{\Phi}[g](s,\zeta)
\dd s, 
\quad 
\zeta\neq0. 
\end{split}
\end{align}
Here $F(\zeta)$ is the nonnegative function
\begin{align}\label{def.F}
\begin{split}
F(\zeta)
:=
\int_{1}^{\infty}
s^{-\frac{\gamma}{2}+1} K_1(|\zeta|s) 
K_{\frac{\gamma}2+1}(|\zeta|s)
\dd s, 
\quad 
\zeta\neq0, 
\end{split}
\end{align}
of which quantitative estimates are given in Lemma \ref{lem.est.F}.

In the following proof, Lemma \ref{lem.MBF} in Appendix \ref{app.lem} will be used without any reference. Let us estimate the right-hand side of \eqref{rep1.proof.lem.vel.L1}. First we consider the second term $\hat{\Phi}[g]$ since the first term depends on it. Using \eqref{rep3.proof.lem.vel.L1}, we start from
\begin{align*}
\begin{split}
&|r^{-\frac{\gamma}{2}} 
K_{\frac{\gamma}2+1}(|\zeta|r) 
M[g](s,\zeta)| \\
&\lesssim
|\zeta|
r^{-\frac{\gamma}{2}}
K_{\frac{\gamma}2+1}(|\zeta|r)
s^{\frac{\gamma}2+1}
\Big(
I_{\frac{\gamma}2+1}(|\zeta|s)
+ I_{\frac{\gamma}2}(|\zeta|s)
\Big)
|g(s,\zeta)|, \\
&|r^{-\frac{\gamma}{2}}
I_{\frac{\gamma}2+1}(|\zeta|r) 
N[g](s,\zeta)| \\
&\lesssim
|\zeta|
r^{-\frac{\gamma}{2}}
I_{\frac{\gamma}2+1}(|\zeta|r)
s^{\frac{\gamma}2+1}
\Big(
K_{\frac{\gamma}2+1}(|\zeta|s)
+ K_{\frac{\gamma}2}(|\zeta|s)
\Big)
|g(s,\zeta)|. 
\end{split}
\end{align*}
When $0<|\zeta|\le1$, we have 
\begin{align}\label{est1.proof.lem.vel.L1}
\begin{split}
&|r^{-\frac{\gamma}{2}} 
K_{\frac{\gamma}2+1}(|\zeta|r) 
M[g](s,\zeta)| \\
&\lesssim
\left\{
\begin{array}{ll}
r^{-\gamma-1} 
s^{\gamma+1}
(|\zeta|s+1)
|g(s,\zeta)|
&\mbox{if}\ 
s \le r \le |\zeta|^{-1} \\ [5pt]
|\zeta|^{\frac{\gamma}{2}+\frac12}
r^{-\frac{\gamma}{2}-\frac12}
e^{-|\zeta|r} 
s^{\gamma+1}
(|\zeta|s+1)
|g(s,\zeta)|
&\mbox{if}\ 
s \le |\zeta|^{-1} \le r \\ [5pt]
r^{-\frac{\gamma}{2}-\frac12}
s^{\frac{\gamma}{2}+\frac12} e^{-|\zeta|(r-s)}
|g(s,\zeta)|
&\mbox{if}\ 
|\zeta|^{-1} \le s \le r 
\end{array}\right. \\
&\lesssim
\left\{
\begin{array}{ll}
r^{-\gamma-1} 
s^{\gamma-\rho}
\|g(\zeta)\|_{L^\infty_{\rho+1}}
&\mbox{if}\ 
s \le r \le |\zeta|^{-1} \\ [5pt]
|\zeta|^{\gamma+1}
e^{-|\zeta|r} 
s^{\gamma-\rho}
\|g(\zeta)\|_{L^\infty_{\rho+1}}
&\mbox{if}\ 
s \le |\zeta|^{-1} \le r \\ [5pt]
s^{-\rho-1} e^{-|\zeta|(r-s)}
\|g(\zeta)\|_{L^\infty_{\rho+1}}
&\mbox{if}\ 
|\zeta|^{-1} \le s \le r 
\end{array}\right. 
\end{split}
\end{align}
and 
\begin{align}\label{est2.proof.lem.vel.L1}
\begin{split}
&
|r^{-\frac{\gamma}{2}}
I_{\frac{\gamma}2+1}(|\zeta|r) 
N[g](s,\zeta)| \\
&\lesssim
\left\{
\begin{array}{ll}
|\zeta| r
(|\zeta|s+1)
|g(s,\zeta)|
&\mbox{if}\ 
r \le s \le |\zeta|^{-1} \\ [5pt]
|\zeta|^{\frac{\gamma}2+\frac32} r
s^{\frac{\gamma}2+\frac12}
e^{-|\zeta|s} 
|g(s,\zeta)|
&\mbox{if}\ 
r \le |\zeta|^{-1} \le s \\ [5pt]
r^{-\frac{\gamma}2-\frac12} 
s^{\frac{\gamma}2+\frac12} e^{|\zeta|(r-s)}
|g(s,\zeta)|
&\mbox{if}\ 
|\zeta|^{-1} \le r \le s
\end{array}\right. \\
&\lesssim
\left\{
\begin{array}{ll}
s^{-\rho-1}
\|g(\zeta)\|_{L^\infty_{\rho+1}}
&\mbox{if}\ 
r \le s \le |\zeta|^{-1} \\ [5pt]
|\zeta|^{\frac{\gamma}2+\frac32}
s^{\frac{\gamma}2-\rho+\frac12}
e^{-|\zeta|s} 
\|g(\zeta)\|_{L^\infty_{\rho+1}}
&\mbox{if}\ 
r \le |\zeta|^{-1} \le s \\ [5pt]
r^{-\frac{\gamma}2-\frac12} 
s^{\frac{\gamma}2-\rho-\frac12} e^{|\zeta|(r-s)}
\|g(\zeta)\|_{L^\infty_{\rho+1}} 
&\mbox{if}\ 
|\zeta|^{-1} \le r \le s. 
\end{array}\right.
\end{split}
\end{align}
When $|\zeta|\ge 1$, we have
\begin{align}\label{est3.proof.lem.vel.L1}
\begin{split}
|r^{-\frac{\gamma}{2}} 
K_{\frac{\gamma}2+1}(|\zeta|r) 
M[g](s,\zeta)| 
&\lesssim
r^{-\frac{\gamma}{2}-\frac12}
s^{\frac{\gamma}{2}+\frac12} e^{-|\zeta|(r-s)}
|g(s,\zeta)| \\
&\lesssim
s^{-\rho-1} e^{-|\zeta|(r-s)}
\|g(\zeta)\|_{L^\infty_{\rho+1}},
\quad
s \le r 
\end{split}
\end{align}
and
\begin{align}\label{est4.proof.lem.vel.L1}
\begin{split}
|r^{-\frac{\gamma}{2}}
I_{\frac{\gamma}2+1}(|\zeta|r) 
N[g](s,\zeta)| 
&\lesssim 
r^{-\frac{\gamma}2-\frac12} 
s^{\frac{\gamma}2+\frac12} e^{|\zeta|(r-s)}
|g(s,\zeta)| \\
&\lesssim
r^{-\frac{\gamma}2-\frac12} 
s^{\frac{\gamma}2-\rho-\frac12} e^{|\zeta|(r-s)}
\|g(\zeta)\|_{L^\infty_{\rho+1}}, 
\quad 
r \le s. 
\end{split}
\end{align}
Then we see from \eqref{est1.proof.lem.vel.L1}--\eqref{est2.proof.lem.vel.L1} that, when $0<|\zeta|\le 1$, 
\begin{align}\label{est5.proof.lem.vel.L1}
\begin{split}
&\int_1^r
|r^{-\frac{\gamma}{2}} 
K_{\frac{\gamma}2+1}(|\zeta|r) 
M[g](s,\zeta)|
\dd s 
+ \int_r^\infty
|r^{-\frac{\gamma}{2}}
I_{\frac{\gamma}2+1}(|\zeta|r) 
N[g](s,\zeta)|
\dd s \\
&\lesssim
\left\{
\begin{array}{ll}
(r^{-\rho} + |\zeta|^{\rho})
\|g(\zeta)\|_{L^\infty_{\rho+1}}
&\mbox{if}\ r \le |\zeta|^{-1} \\ [5pt]
(
|\zeta|^{\rho} e^{-|\zeta|r} 
+ |\zeta|^{-1} r^{-\rho-1}
)
\|g(\zeta)\|_{L^\infty_{\rho+1}}
&\mbox{if}\ r\ge |\zeta|^{-1} 
\end{array}\right.
\end{split}
\end{align}
and from \eqref{est3.proof.lem.vel.L1}--\eqref{est4.proof.lem.vel.L1} that, when $|\zeta|\ge 1$, 
\begin{align}\label{est6.proof.lem.vel.L1}
\begin{split}
&\int_1^r
|r^{-\frac{\gamma}{2}} 
K_{\frac{\gamma}2+1}(|\zeta|r) 
M[g](s,\zeta)|
\dd s
+ \int_r^\infty
|r^{-\frac{\gamma}{2}}
I_{\frac{\gamma}2+1}(|\zeta|r) 
N[g](s,\zeta)|
\dd s \\
&\lesssim
|\zeta|^{-1} r^{-\rho-1} \|g(\zeta)\|_{L^\infty_{\rho+1}}.
\end{split}
\end{align}
Now the estimates \eqref{est5.proof.lem.vel.L1}--\eqref{est6.proof.lem.vel.L1} lead to
\begin{align}\label{est.Phi.proof.lem.vel.L1}
\begin{split}
|\hat{\Phi}(r,\zeta)|
\lesssim
\left\{
\begin{array}{ll}
r^{-\rho} \|g(\zeta)\|_{L^\infty_{\rho+1}}
&\mbox{if}\ 0 < |\zeta| \le1 \\ [5pt]
|\zeta|^{-1} r^{-\rho-1}
\|g(\zeta)\|_{L^\infty_{\rho+1}}
&\mbox{if}\ |\zeta| \ge 1.  
\end{array}\right.
\end{split}
\end{align}
Next we estimate the first term in the right-hand side of \eqref{rep1.proof.lem.vel.L1}. When $0<|\zeta|\le1$, we have
\begin{align}\label{est7.proof.lem.vel.L1}
\begin{split}
&|c[g](\zeta)| \\
&\lesssim
|F(\zeta)|^{-1}
\bigg(
\int_{1}^{|\zeta|^{-1}}
+ \int_{|\zeta|^{-1}}^{\infty}
\bigg)
s K_1(|\zeta|s) |\hat{\Phi}[g](s,\zeta)|
\dd s \\
&\lesssim
|\zeta|^{\frac{\gamma}{2}+2}
\|\hat{\Phi}(\zeta)\|_{L^\infty_{\rho}}
\bigg(
\int_{1}^{|\zeta|^{-1}}
|\zeta|^{-1} s^{-\rho}
\dd s
+ \int_{|\zeta|^{-1}}^{\infty}
|\zeta|^{-\frac12} s^{-\rho+\frac12} e^{-|\zeta|s}
\dd s
\bigg) \\
&\lesssim
|\zeta|^{\frac{\gamma}{2}+1}
\|g(\zeta)\|_{L^\infty_{\rho+1}}. 
\end{split}
\end{align}
When $|\zeta|\ge1$, we have 
\begin{align}\label{est8.proof.lem.vel.L1}
\begin{split}
|c[g](\zeta)|
&\lesssim
|F(\zeta)|^{-1}
\int_{1}^{\infty}
s K_1(|\zeta|s) |\hat{\Phi}[g](s,\zeta)|
\dd s \\
&\lesssim
|\zeta|^{\frac12} e^{2|\zeta|}
\|\hat{\Phi}(\zeta)\|_{L^\infty_{\rho}}
\int_{1}^{\infty}
s^{-2\rho+\frac32} e^{-|\zeta|s}
\dd s \\
&\lesssim
|\zeta|^{-\frac12} e^{|\zeta|}
\|g(\zeta)\|_{L^\infty_{\rho+1}}. 
\end{split}
\end{align}
Then we see from \eqref{est7.proof.lem.vel.L1} that, when $0<|\zeta|\le 1$, 
\begin{align}\label{est9.proof.lem.vel.L1}
\begin{split}
|c[g](\zeta)
r^{-\frac{\gamma}{2}}
K_{\frac{\gamma}2+1}(|\zeta|r)| 
\lesssim
\left\{
\begin{array}{ll}
r^{-\gamma-1}
\|g(\zeta)\|_{L^\infty_{\rho+1}}
&\mbox{if}\ 
r \le |\zeta|^{-1} \\ [5pt]
|\zeta|^{\frac{\gamma}{2}+\frac12} 
r^{-\frac{\gamma}{2}-\frac12}
e^{-|\zeta|r}
\|g(\zeta)\|_{L^\infty_{\rho+1}} 
&\mbox{if}\ 
r \ge |\zeta|^{-1} 
\end{array}\right.
\end{split}
\end{align}
and from \eqref{est8.proof.lem.vel.L1} that, when $|\zeta|\ge 1$, 
\begin{align}\label{est10.proof.lem.vel.L1}
\begin{split}
|c[g](\zeta)
r^{-\frac{\gamma}{2}}
K_{\frac{\gamma}2+1}(|\zeta|r)|
\lesssim
|\zeta|^{-1} 
r^{-\frac{\gamma}2-\frac12}
e^{-|\zeta|(r-1)}
\|g(\zeta)\|_{L^\infty_{\rho+1}}. 
\end{split}
\end{align}
Now the estimates \eqref{est9.proof.lem.vel.L1}--\eqref{est10.proof.lem.vel.L1} lead to
\begin{align}\label{est.crK.proof.lem.vel.L1}
\begin{split}
&|c[g](\zeta)
r^{-\frac{\gamma}{2}}
K_{\frac{\gamma}2+1}(|\zeta|r)| \\
&\lesssim
\left\{
\begin{array}{ll}
r^{-\gamma-1} \|g(\zeta)\|_{L^\infty_{\rho+1}}
&\mbox{if}\ 0 < |\zeta| \le1 \\ [5pt]
|\zeta|^{-1} 
r^{-\frac{\gamma}2-\frac12}
e^{-|\zeta|(r-1)}
\|g(\zeta)\|_{L^\infty_{\rho+1}}
&\mbox{if}\ |\zeta| \ge 1 
\end{array}\right. \\
&\lesssim
\left\{
\begin{array}{ll}
r^{-\rho} \|g(\zeta)\|_{L^\infty_{\rho+1}}
&\mbox{if}\ 0 < |\zeta| \le1 \\ [5pt]
|\zeta|^{-1} r^{-\rho-1}
\|g(\zeta)\|_{L^\infty_{\rho+1}}
&\mbox{if}\ |\zeta| \ge 1. 
\end{array}\right.
\end{split}
\end{align}

Thanks to \eqref{est.Phi.proof.lem.vel.L1} and \eqref{est.crK.proof.lem.vel.L1}, the solution $\hat{\omega}$ in \eqref{rep1.proof.lem.vel.L1} is estimated as 
\begin{align}\label{est.omega.proof.lem.vel.L1}
\begin{split}
\|\hat{\omega}\|_{L^1(L^\infty_{\rho})} 
\lesssim
\|g\|_{L^1(L^\infty_{\rho+1})}. 
\end{split}
\end{align}
Let $\mathcal{V} = \mathcal{V}[\omega]$ denote the Biot-Savart law \eqref{def.ABSL}. Corollary \ref{cor.est.BSL} shows that 
\begin{align}\label{est.V.proof.lem.vel.L1}
\begin{split}
\|(\partial_r,\partial_z)^{j} \mathcal{V}\|_{\opFX(L^\infty_{\rho-1+j})}
\lesssim 
\|\omega\|_{\opFX(L^\infty_{\rho})} 
=\|\hat{\omega}\|_{L^1(L^\infty_{\rho})}
\lesssim
\|g\|_{L^1(L^\infty_{\rho+1})}, 
\quad j=0,1. 
\end{split}
\end{align}

Set $v=\mathcal{V}\in \opFX(\R,L^\infty_{\rho-1})^3\cap W^{1,2}_{{\rm loc}}(\overline{\Omega})^3$. We claim that $v$ is a weak solution of \eqref{eq.LP1} for some associated pressure $q\in W^{1,2}_{{\rm loc}}(\overline{\Omega})$. Indeed, the assumption $g\in C^{\infty}_0(\Omega)^3$ implies that $v$ is smooth on $\Omega$ by elliptic regularity for \eqref{eq.LP1}. Moreover, we see that 
\begin{align*}
\begin{split}
-\Delta v - V\times\oprot v - f 
\in L^2_{{\rm loc}}(\overline{\Omega})^3 
\end{split}
\end{align*}
and, combined with the definition of $\mathcal{V}$, that 
\begin{align*}
\begin{split}
\oprot(
-\Delta v - V\times\oprot v - f
)
=0. 
\end{split}
\end{align*}
Then Proposition \ref{prop.pot} implies that there is a pressure $q\in W^{1,2}_{{\rm loc}}(\overline{\Omega})$ such that
\begin{align}\label{eq.proof.lem.vel.L1}
\begin{split}
\langle -\Delta v - V\times\oprot v + \nabla q, \varphi\rangle
=
\langle f, \varphi\rangle, 
\quad 
\varphi\in C^\infty_{0,\sigma}(\Omega). 
\end{split}
\end{align}
Thus the claim follows. Now Lemma \ref{lem.vel.L1} is proved if $g$ is assumed to belong to $C^{\infty}_0(\Omega)^3$.

Next we consider the general case: let $g\in L^1(\R,L^\infty_{\rho+1})^3$ be axisymmetric and have no swirl. Put $g$ in the formula \eqref{rep1.proof.lem.vel.L1}--\eqref{rep4.proof.lem.vel.L1}. Then the Biot-Savart law $v:=\mathcal{V}$ satisfies \eqref{est.lem.vel.L1}. Examining the estimates above and the ones in the proof of Lemma \ref{lem.strmfunc.L1}, one has 
\begin{align}\label{est11.proof.lem.vel.L1}
\begin{split}
\|\Phi\|_{L^1(L^\infty_{2})}
+ \|\hat{\omega}\|_{L^1(L^\infty_{2})}
+ \|\widehat{\nabla v}\|_{L^1(L^{2})}
\lesssim
\int_{\R} 
\bigg(
\int_{1}^{\infty} s^4 |g(s,\zeta)|^2 \dd s
\bigg)^{\frac12}
\dd \zeta. 
\end{split}
\end{align}
To show that $v$ is a weak solution of \eqref{eq.LP1}, we take $\{g_n\}_{n=1}^{\infty}\subset C^\infty_0(\Omega)^3$ such that 
$$
\lim_{n\to\infty}
\int_{\R} 
\bigg(
\int_{1}^{\infty} s^4 |g(s,\zeta)-g_n(s,\zeta)|^2 \dd s
\bigg)^{\frac12}
\dd \zeta
=0. 
$$
Let $\omega_n$ be given by \eqref{rep1.proof.lem.vel.L1}--\eqref{rep4.proof.lem.vel.L1} with $g$ replaced by $g_n$ and $v_n:=\mathcal{V}[\omega_n]$ denote a smooth solution of \eqref{eq.LP1}. Also, let $q_n\in W^{1,2}_{{\rm loc}}(\overline{\Omega})$ be an associated pressure. From \eqref{est11.proof.lem.vel.L1}, we have
$$
\lim_{n\to\infty}
\|\widehat{\nabla v} - \widehat{\nabla v_n}\|_{L^1(L^{2})}
=0.
$$
Thus we see from \eqref{eq.proof.lem.vel.L1} that 
\begin{align*}
\begin{split}
&
\langle
\nabla v, \nabla \varphi 
\rangle
- \langle 
V\times\oprot v + f, 
\varphi 
\rangle \\
&= \lim_{n\to\infty} 
\langle 
-\Delta v_{n} - V\times\oprot v_n - f_n, 
\varphi 
\rangle \\
&= \lim_{n\to\infty} 
\langle 
\nabla q_{n}, 
\varphi 
\rangle = 0, 
\quad 
\varphi\in C^\infty_{0,\sigma}(\Omega), 
\end{split}
\end{align*}
where integration by parts is used in the first and last equalities. Hence $v$ is a weak solution of \eqref{eq.LP1} satisfying the desired estimate \eqref{est.lem.vel.L1}. This completes the proof of Lemma \ref{lem.vel.L1}.
\end{proof}
%%%%%
%

%
%%%%%%%%%%
\begin{lemma}\label{lem.vel.seq}
Let $\gamma>2$ and $2<\rho<3$ satisfy $\rho\le\gamma$. Let $f\in \opFX(\R,L^\infty_{\rho+1})^3$ have no swirl and be given by, for some $\{a_m\}\in l^1(\Z,L^\infty_{\rho+1})^3$, 
$$
\hat{f} = \sum_{m\in\Z} a_m(r,\theta) \delta_m. 
$$
There is a weak solution $v\in \opFX(\R,L^\infty_{\rho-1})^3\cap W^{1,2}_{{\rm loc}}(\overline{\Omega})^3$ of \eqref{eq.LP1} having the form 
\begin{align}\label{eq.lem.vel.seq}
\begin{split}
\hat{v} = \sum_{m\in\Z} \hat{v}[a_m](r,\theta) \delta_m
\end{split}
\end{align}
and satisfying 
\begin{align}\label{est.lem.vel.seq}
\begin{split}
\|
(\partial_r,\partial_z)^{j} v
\|_{\opFX(L^\infty_{\rho-1+j})}
\lesssim
\|\{a_m\}\|_{l^1(L^\infty_{\rho+1})}, 
\quad j=0,1.
\end{split}
\end{align}
The implicit constant depends on $\gamma,\rho$.
\end{lemma}
%%%%%%%%%%
%
%
%%%%%
\begin{proof}
The plan of the proof is similar to that for Lemma \ref{lem.vel.L1} based on the Biot-Savart law \eqref{def.ABSL}. Indeed, one can easily prove the lemma if $a_0=0$ by following the proof of Lemma \ref{lem.vel.L1}. Thus we may only consider the case where $\hat{f}=a_0\delta_0$. Set $\omega:=(\oprot v)^\theta$.

First we assume that $a_0\in C^{\infty}_0(D)^3$ is axisymmetric and has no swirl. Since 
$$
\widehat{(\oprot f)}^\theta(s,\zeta) 
= \widehat{(\oprot a_0)}^\theta(s,\zeta) 
= - \frac{\dd a_0^z}{\dd s}(s) \delta_{0} 
$$
and $\hat{\omega}=\omega[a_0](r)\delta_0$ solves \eqref{eq.ode.vol}, the formula \eqref{rep.omega1}--\eqref{rep.omega4} shows that 
\begin{align}\label{rep1.proof.lem.vel.seq}
\begin{split}
\omega[a_0](r)
&= c[a_0] r^{-\gamma-1}
+ r^{-\gamma-1} \int_{1}^{r} s^{\gamma+1} a_0^z(s) \dd s, \\
c[a_0]
&:= -\gamma
\int_{1}^{\infty} s^{-\gamma-1} 
\int_{1}^{s} t^{\gamma+1} a_0^z(t) \dd t \dd s. 
\end{split}
\end{align}
Here integration by parts is used in the first line. The constant $c_0[a_0]$ has been chosen so that $d[b_0]$ in Lemma \ref{lem.strmfunc.seq} with $b_0=\omega[a_0]$ is zero. Notice that $c[a_0]$ can be written as 
\begin{align}\label{rep2.proof.lem.vel.seq}
c[a_0]
= -\int_{1}^{\infty} t a_0^z(t) \dd t. 
\end{align}
Then, from \eqref{rep1.proof.lem.vel.seq}--\eqref{rep2.proof.lem.vel.seq} and 
$$
r^{-\gamma-1} \int_{1}^{r} s^{\gamma-\rho} \dd s
\lesssim
r^{-\rho}, 
\qquad
\int_{1}^{\infty} t^{-\rho} \dd t
\lesssim
1, 
$$
we see that 
\begin{align}\label{est.omega.proof.lem.vel.seq}
\begin{split}
\|\omega\|_{\opFX(L^\infty_{\rho})} 
\lesssim
\|a_0\|_{L^\infty_{\rho+1}}. 
\end{split}
\end{align}
Let $\mathcal{V} = \mathcal{V}[\omega]$ denote the Biot-Savart law \eqref{def.ABSL}. Corollary \ref{cor.est.BSL} shows that 
\begin{align}\label{est.V.proof.lem.vel.seq}
\begin{split}
\|(\partial_r,\partial_z)^{j} \mathcal{V}\|_{\opFX(L^\infty_{\rho-1+j})}
\lesssim
\|a_0\|_{L^\infty_{\rho+1}}, 
\quad j=0,1. 
\end{split}
\end{align}
Then a similar argument as in the proof of Lemma \ref{lem.vel.L1} leads to that $v=\mathcal{V}$ is a weak solution of \eqref{eq.LP1} with some pressure $q$, and moreover, gives the proof for the general case $a_0\in (L^\infty_{\rho+1})^3$. We omit the details for simplicity. The proof is complete. 
\end{proof}
%%%%%
%

The following lemma implies the uniqueness of weak solutions for \eqref{eq.LP1}. 
%
%%%%%%%%%%
\begin{lemma}\label{lem.uniqueness}
Let $2<\rho<3$ and $v\in \opFX(\R,L^\infty_{\rho-1})^3\cap W^{1,2}_{{\rm loc}}(\overline{\Omega})^3$ be a weak solution of \eqref{eq.LP1} with $f^r=f^z=0$. Suppose that $\omega := (\oprot v)^\theta\in \opFX(\R,L^\infty_{\rho})$. Then $v=0$. 
\end{lemma}
%%%%%%%%%%
%
%
%%%%%
\begin{proof}
By elliptic regularity for \eqref{eq.LP1}, we see that $v$ is smooth on $\Omega$ and that $\omega:=(\oprot v)^\theta$ solves \eqref{eq.lem.eq.vol} with $(\oprot f)^\theta=0$ in the classical sense. Moreover, the Fourier transform $\hat{\omega}$ solves \eqref{eq.ode.vol} with $\widehat{(\oprot f)}^\theta=0$. Since $\omega \in \opX(\R,L^\infty_{\rho})$, the formula \eqref{rep.omega1}--\eqref{rep.omega4} shows that
\begin{align}\label{rep1.proof.lem.uniqueness}
\begin{split}
\hat{\omega}
= h(r,\zeta) \dd \Lambda 
+ \sum_{m\in\Z} b_m(r) \delta_m, 
\end{split}
\end{align}
where, for some function $c(\zeta)$ and series $\{c_m\}_{m\in\Z}$, 
\begin{align}\label{rep2.proof.lem.uniqueness}
\begin{split}
h(r,\zeta)
&:=
c(\zeta) r^{-\frac{\gamma}{2}} K_{\frac{\gamma}2+1}(|\zeta|r), 
\quad \zeta\neq0, \\
b_m(r)
&:=
\left\{
\begin{array}{ll}
c_0 r^{-\gamma-1}, 
&\mbox{if}\ m=0\\ [5pt]
c_m r^{-\frac{\gamma}{2}} K_{\frac{\gamma}2+1}(|m|r) 
&\mbox{if}\ m\neq0. 
\end{array}\right.
\end{split}
\end{align}

Then Proposition \ref{prop.BSL} shows that $v=\mathcal{V}[\omega]$ with the Biot-Savart law \eqref{def.ABSL}. Moreover, due to the assumption on $v$, we see that $d[h]=0$ and $d[b_m]=0$ for all $m\in \Z$. Hence, by the definition of $d[\,\cdot\,]$ in Lemmas \ref{lem.strmfunc.L1}--\ref{lem.strmfunc.seq}, we find that $c(\zeta)$ and $\{c_m\}_{m\in\Z}$ satisfy 
$$
c(\zeta) F(\zeta)=0, 
\quad \zeta\neq0
$$
and 
$$
\frac{c_0}{\gamma}=0 
\qquad \text{and} \qquad
c_m F(m)=0, 
\quad m\in \Z. 
$$
Here $F(\zeta)$ is the function defined in \eqref{def.F}. Since $F(\zeta)$ is nonnegative, these conditions are equivalent to $c(\zeta)=0$ and $c_m=0$ for all $m\in\Z$. Hence we obtain $\hat{\omega}=0$ because of \eqref{rep1.proof.lem.uniqueness}--\eqref{rep2.proof.lem.uniqueness} and hence $\omega=0$. Then the assertion follows from $v=\mathcal{V}[\omega]=0$. 
\end{proof}
%%%%%
%

%
%%%%%
\begin{proofx}{Proposition \ref{prop.LP1}}
Denote 
\begin{align}
\hat{f}
&= g(r,\theta,\zeta) \dd \Lambda 
+ \sum_{m\in\Z} a_m(r,\theta) \delta_m. 
\end{align}

\noindent ({\bf Existence}) 
Lemmas \ref{lem.vel.L1}--\ref{lem.vel.seq} show that the inverse Fourier transform of 
$$
\hat{v}
= \hat{v}[g](r,\theta,\zeta) \dd \Lambda 
+ \sum_{m\in\Z} \hat{v}[a_m](r,\theta) \delta_m, 
$$
where $\hat{v}[g]$ and $\hat{v}[a_m]$ are the vector fields in the lemmas, is a weak solution of \eqref{eq.LP1} in $\opFX(\R,L^\infty_{\rho-1})^3$. Moreover, $v$ satisfies the estimate \eqref{est.prop.LP1} thanks to \eqref{est.lem.vel.L1} and \eqref{est.lem.vel.seq}.

\noindent ({\bf Uniqueness})
This is implied by Lemma \ref{lem.uniqueness}. This completes the proof. 
\end{proofx}
%%%%%
%

%%%%%%%%%%
%%%%%%%%%%
\subsection{Estimate of swirl}
\label{subsec.swirl}
%%%%%%%%%%
%%%%%%%%%%

We consider the single equation \eqref{eq.LP2} for the swirl $v^\theta$ in this subsection.

The Fourier transform $\hat{v}^\theta=\hat{v}^\theta(r,\zeta)$ solves the boundary value problem
\begin{align}\label{eq.ode.swirl}
-\frac{\dd^2 \hat{v}^\theta}{\dd r^2}
- \frac{1+\gamma}{r} \frac{\dd \hat{v}^\theta}{\dd r} 
+ \Big(\zeta^2 + \dfrac{1-\gamma}{r^2} \Big) \hat{v}^\theta
= \hat{f}^\theta, 
\quad r>1 
\end{align}
with 
\begin{align}\label{bc.ode.swirl}
\hat{v}^\theta(1) = 0, 
\qquad
\lim_{r\to\infty} \hat{v}^\theta(r) = 0. 
\end{align}
The linearly independent solutions of the homogeneous equation are
$$
\left\{
\begin{array}{ll}
r^{-\frac{\gamma}2} 
K_{|\frac{\gamma}2-1|}(|\zeta|r) 
\mkern9mu \text{and} \mkern9mu 
r^{-\frac{\gamma}2} 
I_{|\frac{\gamma}2-1|}(|\zeta|r) 
&\mbox{if}\ \zeta\neq0 \\ [5pt]
r^{-\gamma+1}
\mkern9mu \text{and} \mkern9mu 
r^{-1}
&\mbox{if}\ \zeta=0 
\end{array}
\right. 
$$
with the Wronskian 
$$
\left\{
\begin{array}{ll}
r^{-\frac{\gamma}2-1}
&\mbox{if}\ \zeta\neq0 \\ [5pt]
(\gamma-2)r^{-\gamma-1}
&\mbox{if}\ \zeta=0. 
\end{array}
\right. 
$$
Thus, if the Fourier transform of $f^\theta$ is given by 
$$
\hat{f}^\theta
= g^\theta(r,\zeta) \dd \Lambda 
+ \sum_{m\in\Z} a^\theta_m(r) \delta_m, 
$$
the solution $\hat{v}^\theta$ of \eqref{eq.ode.swirl} belonging to $\bigcup_{\rho\ge0} \opX(\R,L^\infty_{\rho})$ has the form 
\begin{align}\label{rep.vtheta1}
\begin{split}
\hat{v}^\theta(r)
=
\int_1^\infty 
(\hat{f}^\theta\lfloor \sigma_7)(r,s) \dd s 
+ \int_1^r 
(\hat{f}^\theta\lfloor \sigma_8)(r,s) \dd s 
+ \int_r^\infty 
(\hat{f}^\theta\lfloor \sigma_9)(r,s) \dd s. 
\end{split}
\end{align}
Here $\hat{f}^\theta\lfloor \sigma_i$ denotes the Borel measure 
\begin{align}\label{rep.vtheta2}
\begin{split}
\hat{f}^\theta\lfloor \sigma_i(r,s,\mathcal{O})
&=
\int_{\mathcal{O}}
\sigma_i(r,s,\zeta) g^\theta(s,\zeta) \dd \Lambda \\
&\quad 
+ \sum_{m\in\Z\cap\mathcal{O}}
\sigma_i(r,s,m) a^\theta_m(s) \delta_m, 
\quad \mathcal{O}\in \mathcal{B}(\R) 
\end{split}
\end{align}
and $\sigma_7, \sigma_8, \sigma_9$ are the functions 
\begin{align}\label{rep.vtheta3}
\begin{split}
\sigma_7(r,s,\zeta) 
&= 
\left\{
\begin{array}{ll}
-\dfrac{I_{|\frac{\gamma}2-1|}(|\zeta|)}{K_{|\frac{\gamma}2-1|}(|\zeta|)} 
r^{-\frac{\gamma}{2}}
K_{|\frac{\gamma}2-1|}(|\zeta|r) 
s^{\frac{\gamma}2+1}
K_{|\frac{\gamma}2-1|}(|\zeta|s) 
&\mbox{if}\ \zeta\neq0 \\ [10pt]
\dfrac{1}{\gamma-2}
r^{-\gamma+1}
s^{2}
&\mbox{if}\ \zeta=0
\end{array}
\right. \\
\sigma_8(r,s,\zeta) 
&= 
\left\{
\begin{array}{ll}
r^{-\frac{\gamma}{2}}
K_{|\frac{\gamma}2-1|}(|\zeta|r)
s^{\frac{\gamma}{2}+1}
I_{|\frac{\gamma}2-1|}(|\zeta|s)
&\mbox{if}\ \zeta\neq0 \\ [10pt]
\dfrac{1}{\gamma-2}
r^{-\gamma+1}
s^{\gamma}
&\mbox{if}\ \zeta=0
\end{array}
\right. \\
\sigma_9(r,s,\zeta) 
&= 
\left\{
\begin{array}{ll}
r^{-\frac{\gamma}2}
I_{|\frac{\gamma}2-1|}(|\zeta|r)
s^{\frac{\gamma}2+1}
K_{|\frac{\gamma}2-1|}(|\zeta|s)
&\mbox{if}\ \zeta\neq0 \\ [10pt]
\dfrac{1}{\gamma-2}
r^{-1}s^2
&\mbox{if}\ \zeta=0. 
\end{array}
\right. 
\end{split}
\end{align}
A representation of $\hat{v}^\theta$ is obtained by the inverse Fourier transform of \eqref{rep.vtheta1}--\eqref{rep.vtheta3}.

%
%%%%%%%%%%
\begin{proposition}\label{prop.LP2}
Let $\gamma>2$ and $2<\rho<3$ satisfy $\rho\le\gamma$. Let $f^\theta\in \opFX(\R,L^\infty_{\rho+1})$. There is a unique weak solution $v^\theta\in \opFX(\R,L^\infty_{\rho-1})\cap W^{1,2}_{{\rm loc}}(\overline{\Omega})$ of \eqref{eq.LP2} satisfying 
\begin{align}\label{est.prop.LP2}
\|(\partial_r,\partial_z)^j v^\theta\|_{\opFX(L^\infty_{\rho-1+j})} 
\lesssim
\|f^\theta\|_{\opFX(L^\infty_{\rho+1})}, 
\quad j=0,1. 
\end{align}
The implicit constant depends on $\gamma, \rho$.
\end{proposition}
%%%%%%%%%%
%
%
%%%%%
\begin{proof}
The uniqueness is obvious. The estimate \eqref{est.prop.LP2} can be proved in a procedure similar to that in the proof of Proposition \ref{prop.LP1}. We omit the details for simplicity. 
\end{proof}
%%%%%
%

%%%%%%%%%%%%%%%%%%%%
%%%%%%%%%%%%%%%%%%%%
\section{Proof of Theorem \ref{thm.main}}
\label{sec.pr.thm}
%%%%%%%%%%%%%%%%%%%%
%%%%%%%%%%%%%%%%%%%%

This section is devoted to the proof of Theorem \ref{thm.main}. However, since the proof is standard using the Banach fixed-point theorem, we only present the outline for simplicity.
%
%%%%%%%%%%
\begin{proofx}{Theorem \ref{thm.main}}
Define the Banach space 
\begin{align*}
{\mathcal X}_\rho
=\Big\{w\in \opFX(\R,L^\infty_{\rho-1})^3
~\Big|~ 
w|_{\{1\}\times\R}=0, 
\mkern9mu
(\partial_r,\partial_z) w \in \opFX(\R,L^\infty_{\rho})^{2\times3}
\Big\} 
\end{align*}
with the norm $\|w\|_{{\mathcal X}_\rho}:=
\|w\|_{\opFX(L^\infty_{\rho-1})} + \|(\partial_r,\partial_z) w\|_{\opFX(L^\infty_{\rho})}$. By Proposition \ref{prop.LP}, for any given $f\in\opFX(\R,L^\infty_{\rho+1})^3$ and $w\in {\mathcal X}_{\rho}$, there is a unique weak solution $v_w$ to \eqref{eq.LP} with the components of $-w\cdot\nabla w + f$ as an external force. Moreover, $v_w$ satisfies 
\begin{align*}
\begin{split}
\|v_w\|_{\opFX(L^\infty_{\rho-1+j})} 
&\lesssim
\|-w\cdot\nabla w\|_{\opFX(L^\infty_{\rho+1})} 
+ \|f\|_{\opFX(L^\infty_{\rho+1})} \\
&\lesssim
\|w\|_{\opFX(L^\infty_{\rho-1})} 
\|(\partial_r,\partial_z) w\|_{\opFX(L^\infty_{\rho})} 
+ \|f\|_{\opFX(L^\infty_{\rho+1})}.
\end{split}
\end{align*}
Here Lemma \ref{lem.X.FX} and a trivial inequality $\rho+1\le 2\rho-1$ are used. Thus the mapping $T:{\mathcal X}_\rho\ni w\mapsto v_w\in {\mathcal X}_\rho$ is well-defined. It is easy to check that, if $\|f\|_{\opFX(L^\infty_{\rho+1})}$ and $\delta$ are sufficiently small depending on $\gamma,\rho$, then $T$ is a contraction on the closed ball $\mathcal{B}_\delta\subset{\mathcal X}_\rho$ centered at the origin with radius $\delta$. Hence the existence of weak solutions of \eqref{intro.eq.NP} unique in $\mathcal{B}_\delta$ follows from the Banach fixed-point theorem. The assertions for the Navier-Stokes system \eqref{intro.eq.NS} are obvious. This completes the proof of Theorem \ref{thm.main}. 
\end{proofx}
%%%%%%%%%%
%

%%%%%%%%%%%%%%%%%%%%
%%%%%%%%%%%%%%%%%%%%
\appendix
%%%%%%%%%%%%%%%%%%%%
%%%%%%%%%%%%%%%%%%%%

%%%%%%%%%%%%%%%%%%%%
%%%%%%%%%%%%%%%%%%%%
\section{Auxiliary lemmas}
\label{app.lem}
%%%%%%%%%%%%%%%%%%%%
%%%%%%%%%%%%%%%%%%%%

This appendix collects the lemmas needed in Sections \ref{sec.ABSL} and \ref{sec.LP}. References are \cite{Watbook1944,AARbook1999}.

%%%%%%%%%%
%%%%%%%%%%
\subsection*{Modified Bessel function}
%%%%%%%%%%
%%%%%%%%%%

The modified Bessel function of the first kind $I_\nu(z)$ of order $\nu$ is defined by 
\begin{align}\label{def.I}
\begin{split}
I_\nu(z) 
= 
\Big(\frac{z}{2}\Big)^\nu 
\sum_{k=0}^{\infty} 
\frac{1}{k!\Gamma(\nu+k+1)} \Big(\frac{z}{2}\Big)^{2k}, 
\quad 
z\in \C\setminus (-\infty,0], 
\end{split}
\end{align}
where $\Gamma(z)$ denotes the Gamma function, the second kind $K_\nu(z)$ of order $\nu\notin\Z$ is by 
\begin{align}\label{def.K}
K_\nu(z) 
= 
\frac{\pi}{2}
\frac{I_{-\nu}(z) - I_\nu(z)}{\sin(\nu \pi)}, 
\quad 
z\in \C\setminus (-\infty,0],
\end{align}
and $K_n(z)$ of order $n\in\Z$ is by the limit of $K_\nu(z)$ in \eqref{def.K} as $\nu\to n$. It is well known that $I_\nu(z)$ grows exponentially and $K_\nu(z)$ decays exponentially as $z\to\infty$ in $\{\Re z>0\}$; see \cite[Section 4.12]{AARbook1999}. Moreover, $K_\nu(z)$ and $I_\nu(z)$ are linearly independent solutions of 
$$
-\frac{\dd^2 \omega}{\dd z^2} 
- \frac{1}{z} \frac{\dd \omega}{\dd z} 
+ \Big(1+\frac{\nu^2}{z^2}\Big) \omega = 0 
$$
and the Wronskian is
\begin{align}\label{def.W}
\det 
\left(
\begin{array}{cc}
K_\nu(z) 
& I_\nu(z) \\ [5pt]
\displaystyle{\frac{\dd K_\nu}{\dd z}(z)} 
& \displaystyle{\frac{\dd I_\nu}{\dd z}(z)}
\end{array}
\right)
=
\frac1z. 
\end{align}
It is known that $I_\nu(z)$ and $K_\nu(z)$ satisfy
\begin{align}\label{eq.rel}
\begin{split}
\frac{\dd I_\nu}{\dd z}(z)
&
=\frac{\nu}{z} I_\nu(z) + I_{\nu+1}(z)
=-\frac{\nu}{z} I_\nu(z) + I_{\nu-1}(z), \\
\frac{\dd K_\nu}{\dd z}(z) 
&
=\frac{\nu}{z} K_\nu(z) - K_{\nu+1}(z)
=-\frac{\nu}{z} K_\nu(z) - K_{\nu-1}(z). 
\end{split}
\end{align}

The following lemma collects the properties of $I_\nu(r)$ and $K_\nu(r)$ with $\nu>0$ and $r>0$. Each of them easily follows from the formulas that can be found in \cite{Watbook1944,AARbook1999}. 
%
%%%%%%%%%%
\begin{lemma}\label{lem.MBF}
Let $\nu>0$. Then the following hold.
\begin{enumerate}[(1)]
\item\label{item1.lem.MBF}
$I_\nu(r)$ and $K_\nu(r)$ are positive when $r>0$.

\item\label{item2.lem.MBF}
More precisely, when $0<r\le1$, 
\begin{align*}
\begin{split}
K_\nu(r)
\approx
r^{-\nu}, 
\qquad 
I_\nu(r)
\approx
r^{\nu}
\end{split}
\end{align*}
and when $r\ge1$, 
\begin{align*}
\begin{split}
K_\nu(r)
\approx
r^{-\frac12} e^{-r}, 
\qquad
I_\nu(r)
\approx
r^{-\frac12} e^{r}. 
\end{split}
\end{align*}
The implicit constants depend on $\nu$. 
\end{enumerate}
\end{lemma}
%%%%%%%%%%
%

%%%%%%%%%%
%%%%%%%%%%
\subsection*{Estimate of the function in \eqref{def.F}}
%%%%%%%%%%
%%%%%%%%%%

The function $F(\zeta)$ in \eqref{def.F} is estimated as follows. The key to the proof is to utilize the monotonicity of integration due to the positivity of $K_\nu(r)$ with $\nu>0$ and $r>0$. 
%
%%%%%%%%%%
\begin{lemma}\label{lem.est.F}
Let $\gamma>0$. Then 
\begin{align*}
\begin{split}
F(\zeta)
\gtrsim
\left\{
\begin{array}{ll}
|\zeta|^{-\frac{\gamma}{2}-2}
&\mbox{if}\ 0 < |\zeta| \le1 \\ [5pt]
|\zeta|^{-2} e^{-2|\zeta|}
&\mbox{if}\ |\zeta| \ge 1.  
\end{array}\right.
\end{split}
\end{align*}
The implicit constant depends on $\gamma$.
\end{lemma}
%%%%%%%%%%
%
%
%%%%%
\begin{proof}
We see from Lemma \ref{lem.MBF} that, when $0<|\zeta|\le 1$, 
\begin{align*}
\begin{split}
\int_1^{|\zeta|^{-1}}
s^{-\frac{\gamma}{2}+1} 
K_1(|\zeta|s) 
K_{\frac{\gamma}2+1}(|\zeta|s)
\dd s
&\gtrsim
\int_{1}^{|\zeta|^{-1}}
|\zeta|^{-\frac{\gamma}{2}-2}
s^{-\gamma-1} 
\dd s \\
&\gtrsim
|\zeta|^{-\frac{\gamma}{2}-2}, 
\end{split}
\end{align*}
and when $|\zeta|\ge 1$, 
\begin{align*}
\begin{split}
\int_{1}^{\infty}
s^{-\frac{\gamma}{2}+1} 
K_1(|\zeta|s) 
K_{\frac{\gamma}2+1}(|\zeta|s)
\dd s
&\gtrsim
\int_{1}^{\infty}
|\zeta|^{-1}
s^{-\frac{\gamma}{2}}
e^{-2|\zeta|s}
\dd s \\
&\gtrsim
|\zeta|^{-2} e^{-2|\zeta|}. 
\end{split}
\end{align*}
Hence the desired estimate follows. 
\end{proof}
%%%%%
%

%%%%%%%%%%%%%%%%%%%%
%%%%%%%%%%%%%%%%%%%%
\section{Flows under rotation and suction}
\label{app.case.rot}
%%%%%%%%%%%%%%%%%%%%
%%%%%%%%%%%%%%%%%%%%

In this appendix, we discuss the case where $V$ is defined not in \eqref{def.V} but in \eqref{def.V.rot}.

Using the formulas \eqref{formula1} and \eqref{formula4} in Subsection \ref{subsec.op}, we have
\begin{align*}
\begin{split}
V\times\oprot v
&=
V^\theta (\oprot v)^z {\bf e}_r 
- V^r (\oprot v)^z {\bf e}_\theta 
+ \{V^r (\oprot v)^\theta - V^\theta (\oprot v)^r\} {\bf e}_z \\
&=
\frac{\alpha}{r^2} \partial_r (rv^\theta) {\bf e}_r
+ \frac{\gamma}{r^2} \partial_r (rv^\theta) {\bf e}_\theta
+ \Big\{
-\frac{\gamma}{r} (\partial_z v^r - \partial_r v^z) 
+ \frac{\alpha}{r} \partial_z v^\theta
\Big\}
{\bf e}_z. 
\end{split}
\end{align*}
The linearized problem of \eqref{intro.eq.PNP} in the introduction around $V$ is now given by 
\begin{equation}\label{eq.LP.rot}
\left\{
\begin{array}{ll}
-\Big(
\Delta - \dfrac{1}{r^2} 
\Big) v^r 
- \dfrac{\alpha}{r^2} \partial_r (rv^\theta) 
+ \partial_r q = f^r
&\mbox{in}\ (1,\infty)\times\R \\ [10pt]
-\Big(\Delta
+ \dfrac{\gamma}{r} \partial_r
- \dfrac{1-\gamma}{r^2}
\Big) v^\theta 
= f^\theta
&\mbox{in}\ (1,\infty)\times\R \\ [10pt]
-\Delta v^z 
+ \dfrac{\gamma}{r} (\partial_z v^r - \partial_r v^z)
- \dfrac{\alpha}{r} \partial_z v^\theta
+ \partial_z q
= f^z
&\mbox{in}\ (1,\infty)\times\R, 
\end{array}\right.
\end{equation}
where the boundary conditions are omitted for simplicity.

Notice that the equation for the swirl $v^\theta$ does not change from the case $\alpha=0$. 
%
%%%%%%%%%%
\begin{lemma}\label{app.lem.eq.vol}
Let $\alpha,\gamma\in\R$ and let $f\in C^\infty(\Omega)^3$ be axisymmetric and have no swirl. Suppose that $(v,\nabla q)$ is a smooth solution of \eqref{eq.LP.rot}. Then the component $\omega:=(\oprot v)^\theta$ solves 
\begin{align}\label{eq.app.lem.eq.vol}
\begin{split}
-\Big(
\Delta + \frac{\gamma}{r}\partial_r - \dfrac{1+\gamma}{r^2} 
\Big) \omega
= (\oprot f)^\theta + \dfrac{2\alpha}{r^2} \partial_z v^\theta. 
\end{split}
\end{align}
\end{lemma}
%%%%%%%%%%
%
%
%%%%%
\begin{proof}
In a similar manner as in the proof of Lemma \ref{lem.eq.vol}, we see that $\omega$ solves
\begin{align*}
\begin{split}
-\Big(
\Delta + \dfrac{\gamma}{r}\partial_r - \dfrac{1+\gamma}{r^2} 
\Big) \omega
+ \alpha 
\Big\{
-\frac{1}{r^2} \partial_r (r\partial_z v^\theta)
+\partial_r\Big(\dfrac{\partial_z v^\theta}{r}\Big)
\Big\}
= (\oprot f)^\theta, 
\end{split}
\end{align*}
which implies the assertion. 
\end{proof}
%%%%%
%

Therefore, by solving the equation for the swirl $v^\theta$ first and by regarding the right-hand side of \eqref{eq.app.lem.eq.vol} as given data, it is possible to reproduce the proofs in Subsection \ref{subsec.rad.vert}. Indeed, one can obtain the estimate in Proposition \ref{prop.LP} but with linear growth in $|\alpha|$, which concludes Proposition \ref{app.prop} below. It is unclear whether the smallness on $|\alpha|$ can be removed.
%
%%%%%%%%%%
\begin{proposition}\label{app.prop}
Let $\alpha\in\R$ and let $\gamma>2$ and $2<\rho<3$ satisfy $\rho\le\gamma$. If $f\in \opFX(\R,L^\infty_{\rho+1})^3$ and $|\alpha|, \|f\|_{\opFX(L^\infty_{\rho+1})}$ are sufficiently small depending on $\gamma,\rho$, then there is a unique weak solution $v\in \opFX(\R,L^\infty_{\rho-1})^3$ of \eqref{intro.eq.PNP}. Consequently, there is an axisymmetric weak solution $u$ of \eqref{intro.eq.NS} with $b = \alpha {\bf e}_\theta - \gamma {\bf e}_r$ unique in a suitable class satisfying
\begin{align*}
u(r,\theta,z) 
= 
\frac{\alpha}{r} {\bf e}_\theta
- \frac{\gamma}{r} {\bf e}_r
+ O(r^{-\rho+1})
\quad 
\mbox{as}\ \, r\to\infty.
\end{align*}
\end{proposition}
%%%%%%%%%%

%%%%%%%%%%
%%%%%%%%%%
\subsection*{Acknowledgements}
The author is partially supported by JSPS KAKENHI Grant Number JP 20K14345. 
%%%%%%%%%%
%%%%%%%%%%

%%%%%%%%%%%%%%%%%%%%
%%%%%%%%%%%%%%%%%%%%
\addcontentsline{toc}{section}{References}
\bibliography{Ref}
\bibliographystyle{plain}
%%%%%%%%%%%%%%%%%%%%
%%%%%%%%%%%%%%%%%%%%

\medskip

\begin{flushleft}
M. Higaki\\
Department of Mathematics, 
Graduate School of Science, 
Kobe University, 
1-1 Rokkodai, 
Nada-ku, 
Kobe 657-8501, 
Japan.
Email: higaki@math.kobe-u.ac.jp
\end{flushleft}

\medskip

\noindent \today

\end{document}